\documentclass[journal]{IEEEtran}
\usepackage{cite}
\usepackage[pdftex]{graphicx}
\graphicspath{figures/}
\DeclareGraphicsExtensions{.pdf,.jpeg,.png}
\usepackage{amsmath,amssymb,amsthm,mathtools,bm}
\usepackage{amsmath, amsfonts, amsthm, amssymb, mathtools, enumitem}
\usepackage{xcolor, epstopdf, url}
\usepackage{array}
\usepackage{siunitx}

\usepackage{algorithm}
\usepackage{algcompatible}
\renewcommand\algorithmicthen{}

\algnewcommand{\IIf}[1]{\State\algorithmicif\ #1\ \algorithmicthen}
\algnewcommand{\EndIIf}{\unskip\ \algorithmicend\ \algorithmicif}
\algnewcommand{\IElse}{\unskip\ \algorithmicelse \unskip\ }

\usepackage{booktabs,tabularx,makecell,colortbl,multirow}
\newcolumntype{C}{>{\centering\arraybackslash}X}
\newcolumntype{L}{>{\raggedright\arraybackslash}X}
\newcolumntype{R}{>{\raggedleft\arraybackslash}X}

\newcommand{\Image}[1]{\mathop{\mathrm{Im}}\left(#1\right)}
\newcommand{\norm}[1]{\left\lVert #1 \right\rVert}
\newcommand{\dist}[2]{\mathop{\mathrm{dist}}\left(#1, #2\right)}
\newtheorem{proposition}{Proposition}
\newtheorem{remark}{Remark}
\newtheorem{theorem}{Theorem}
\newtheorem{lemma}{Lemma}

\begin{document}
\title{A Globally Convergent Distributed Jacobi Scheme for Block-Structured Nonconvex Constrained Optimization Problems}

\author{Anirudh~Subramanyam,
        Youngdae~Kim,
        Michel~Schanen,
        Fran\c{c}ois~Pacaud,
        and~Mihai~Anitescu,~\IEEEmembership{Member,~IEEE}%
\thanks{The authors are with the Mathematics and Computer Science Division, Argonne National Laboratory, Lemont IL 60439. This work was supported by the U.S. Department of Energy, Office of Science, under contract DE-AC02-06CH11357.}%
}

\maketitle

\begin{abstract}
Motivated by the increasing availability of high-performance parallel computing, we design a distributed parallel algorithm for linearly-coupled block-structured nonconvex constrained optimization problems. Our algorithm performs Jacobi-type proximal updates of the augmented Lagrangian function, requiring only local solutions of separable block nonlinear programming (NLP) problems. We provide a cheap and explicitly computable Lyapunov function that allows us to establish global and local sublinear convergence of our algorithm, its iteration complexity, as well as simple, practical and theoretically convergent rules for automatically tuning its parameters. This in contrast to existing algorithms for nonconvex constrained optimization based on the alternating direction method of multipliers that rely on at least one of the following: Gauss-Seidel or sequential updates, global solutions of NLP problems, non-computable Lyapunov functions, and hand-tuning of parameters. Numerical experiments showcase its advantages for large-scale problems, including the multi-period optimization of a 9000-bus AC optimal power flow test case over 168 time periods, solved on the Summit supercomputer using an open-source Julia code.
\end{abstract}


\IEEEpeerreviewmaketitle

\section{Introduction}

Block-structured nonlinear optimization models are ubiquitous in science and engineering applications.
Some examples include nonlinear model predictive control, multi-stage stochastic programming, supervised machine learning, optimization with differential-algebraic equations, as well as network control with geographically distributed agents.
These models are used in a wide variety of areas including power systems, telecommunications, sensor networks, smart manufacturing, and chemical process systems, to name but a few.
The block structure in several of these models have the following form, which is the subject of the present paper.
\begin{equation}\label{eq:nlp_general}
    \begin{aligned}
        \displaystyle\mathop{\text{minimize}}_{x_1, \ldots, x_T} \;\; & \displaystyle \sum_{t = 1}^T f_t(x_t) \\
        \text{subject to} \;\;& \displaystyle x_t \in X_t, \;\; t \in \{1, 2, \ldots, T\}, \\
        & \displaystyle \sum_{t = 1}^T A_t x_t = b,
    \end{aligned}
\end{equation}
where $X_t$ %
are compact (possibly nonconvex) sets, $A_t \in \mathbb{R}^{m \times n_t}$ are matrices, and $f_t: \mathbb{R}^{n_t} \mapsto \mathbb{R}$ %
are continuously differentiable (possibly nonconvex) functions.

In formulation~\eqref{eq:nlp_general}, the decision variables are grouped into $T$ blocks $(x_1, x_2, \ldots, x_T)$ coupled only via the linear constraints $\sum_{t = 1}^T A_t x_t = b$.
Such linearly coupled block-structured models arise in a wide variety of applications.
For example, consider a discrete-time dynamical system:
\[
x_{t+1} = \phi(x_t, u_t) \; ,
\]
where $x_t \in \mathbb{R}^{n_x}$ and $u_t \in \mathbb{R}^{n_u}$ are the vectors of states and inputs, respectively, and $\phi :\mathbb{R}^{n_x} \times \mathbb{R}^{n_u} \mapsto \mathbb{R}^{n_x}$ is some smooth state transition function.
In such systems, the optimal control problem that must be solved in the context of model predictive control \cite{grune2017nonlinear} can be brought into the form of \eqref{eq:nlp_general} by introducing additional variables $\tilde{x}_{t} \in \mathbb{R}^{n_x}$ and additional constraints:
\[
x_{t+1} = \tilde{x}_{t}, \;\; \tilde{x}_{t} = \phi(x_t, u_t),
\]
and by defining the $t^\mathrm{th}$ block to be $X_t = \{(x_t, u_t, \tilde{x}_t) \in \mathbb{R}^{n_x} \times \mathbb{R}^{n_u} \times \mathbb{R}^{n_x}: \tilde{x}_{t} = \phi(x_t, u_t) \}$.
The $T$ blocks are then coupled via the linear equations $x_{t+1} = \tilde{x}_{t}$, $t=1,2,\ldots,T-1$.

The goal of this paper is to devise a parallel algorithm that can solve instances of formulation \eqref{eq:nlp_general} when either the number of blocks $T$ or the number of variables/constraints in individual blocks $X_t$ is large.
In such cases, memory or storage requirements may prohibit the direct solution of \eqref{eq:nlp_general} using conventional nonlinear programming (NLP) solvers.
Also, structure-agnostic solvers may not be able to exploit any available distributed parallel computing capabilities, resulting in excessively high computational times that can be detrimental in several applications including model predictive control.

\subsection{Literature review}

A popular approach to solving large-scale instances of \eqref{eq:nlp_general} is to decompose the problem iteratively into individual block subproblems that are easier to handle numerically.
This decomposition can be done either at the linear algebra level of an interior point method \cite{gondzio2009exploiting,chiang2014structured,kourounis2018toward} or by utilizing (augmented) Lagrangian functions to solve block-separable dual problems that are derived using (local) convex duality \cite{bertsekas2014constrained,houska2016augmented,hours2016parametric}.
The alternating direction method of multipliers (ADMM) belongs to the latter class of methods and has seen a recent surge in popularity because of its suitability for distributed computation; e.g., see \cite{boyd2011distributed}.
However, the vast majority of provably convergent ADMM approaches for solving \eqref{eq:nlp_general} either exploit convexity in the objective function $f_t$ or constraints $X_t$ or they are tailored for specific instances of \eqref{eq:nlp_general}; e.g., see \cite{li2015global,lin2015global,magnusson2016,hong2016convergence,hong2017linear,melo2017iteration}.

A recent body of literature~\cite{bolte2014proximal,liu2019linearized,wang2019global,jiang2019structured,sun2019two,tang2021fast,yang2020proximal,shi2020penalty,zhu2020first,harwood2021analysis} has analyzed the global and local convergence of ADMM approaches for non-convex instances of formulation \eqref{eq:nlp_general}.
It is well-understood by now (e.g., see \cite{wang2019global}) that for $T > 2$, convergence requires the existence of a block variable  that is constrained only by the linear coupling equation, and whose objective function is globally Lipschitz differentiable.
Moreover, it must be coupled in such a way that it can ``control'' the iterates of the coupling constraints' dual variables (see Remark~\ref{rem:global_convergence} later in the paper).

One way to achieve this is to first introduce an additional block of slack variables $z \in \mathbb{R}^m$ that relax the linear coupling constraints, and then drive $z$ to $0$ by adding a smooth quadratic penalty term in the objective function:
\begin{equation}\label{eq:nlp}
    \begin{aligned}
        \displaystyle\mathop{\text{minimize}}_{x_1, \ldots, x_T, z} \;\;& \displaystyle \sum_{t = 1}^T f_t(x_t) + \frac{\theta}{2} \norm{z}^2 \\
        \text{subject to} \;\;& \displaystyle z \in \mathbb{R}^{m}, \;\; x_t \in X_t, \;\; t \in \{1,2, \ldots, T\}, \\
        & \displaystyle \sum_{t = 1}^T A_t x_t + z = b,
    \end{aligned}
\end{equation}
where the penalty coefficient $\theta$ is a function of the target tolerance $\epsilon > 0$ for satisfying the linear coupling constraints.
In \cite{jiang2019structured}, this idea is exploited to design algorithms that globally converge to $\epsilon$-stationary points of \eqref{eq:nlp_general} in $O(\epsilon^{-6})$ iterations.
Moreover, the algorithms require the global solution of NLP problems~\cite[Remark~3.8]{jiang2019structured}, which can be numerically demanding when the local subproblems are nonconvex.
A similar quadratic penalty idea is adopted in \cite{yang2020proximal} although it requires the augmented Lagrangian function (with respect to all nonlinear and coupling constraints) to satisfy the Kurdyka-\L{}ojasiewicz property \cite{bolte2014proximal}; furthermore, that analysis does not reconcile the penalty formulation \eqref{eq:nlp} with the original formulation \eqref{eq:nlp_general}.
Each iteration linearizes the augmented Lagrangian function around the previous iterate that is then alternately minimized with respect to the $T$ blocks.
Other approaches using local linear and convex approximations of the nonlinear constraints or of the augmented Lagrangian function, and methods establishing convergence under additional assumptions such as coercivity of the objective function, have been proposed in \cite{scutari2016parallel,liu2019linearized,yi2019linear,zhu2020first}.

Instead of using a quadratic penalty, \cite{sun2019two,tang2021fast} replace the objective in~\eqref{eq:nlp} by the augmented Lagrangian function
\begin{equation}\label{eq:two-level-admm-outer-level}
\sum_{t = 1}^T f_t(x_t) + \beta^\top z + \frac{\theta}{2} \norm{z}^2,
\end{equation}
and drive $z$ to $0$ by iteratively updating the Lagrange multipliers $\beta$, where each so-called outer loop iteration uses another inner-level iterative ADMM scheme that minimizes \eqref{eq:two-level-admm-outer-level} subject to the constraints of \eqref{eq:nlp}.
An alternative two-level decomposition for (possibly nonlinear) coupling constraints is proposed in \cite{shi2020penalty} but it depends intimately on randomized block updating of the augmented Lagrangian function which in turn requires locally tight upper bounds of the objective function with respect to each variable block, while fixing the other blocks.

All of the aforementioned algorithms rely on a Gauss-Seidel updating scheme, where the variable blocks are iteratively optimized in a particular sequence.
Specifically, when updating a given variable block (say $x_t$), the values of all other blocks that appear earlier in the sequence (e.g., $x_1, \ldots, x_{t-1}$) must be fixed to their newest values.
The overall convergence relies critically on the sequential nature of this update.
Unfortunately, this makes these algorithms unsuitable for distributed parallel computations.
A common workaround (e.g., see \cite{houska2016augmented,deng2017parallel}) to enable parallel computation is to equivalently reformulate \eqref{eq:nlp_general}, \eqref{eq:nlp} by introducing additional variables $y_t$ as follows:
\begin{equation}\label{eq:nlp_variable_splitting}
    \begin{array}{r@{\;\;}l}
        \displaystyle\mathop{\text{minimize}}_{\substack{x_1, \ldots, x_T\\y_1, \ldots, y_T}} & \displaystyle \sum_{t = 1}^T f_t(x_t) \\
        \text{subject to} & \displaystyle x_t \in X_t, \;\; A_tx_t = y_t, \;\;  t \in \{1,2, \ldots, T\}, \\
        & \displaystyle \sum_{t = 1}^T y_t = b,
    \end{array}
\end{equation}
The augmented Lagrangian function with respect to the last equation is completely decomposable in the $x$-variables:
\begin{equation}\label{eq:nlp_variable_splitting_aug_lag}
\sum_{t = 1}^T f_t(x_t) + \beta^\top \Big(\sum_{t = 1}^T y_t - b \Big) + \frac{\theta}{2} \Big\|\sum_{t = 1}^T y_t - b\Big\|^2 \; .
\end{equation}
Therefore, the reformulation \eqref{eq:nlp_variable_splitting} can be interpreted as a two-block model where the $x$- and $y$-variables constitute respectively the first and second blocks.
Since \eqref{eq:nlp_variable_splitting_aug_lag} admits a closed-form minimization with respect to $y$ (for fixed values of $x$), one can use any of the aforementioned Gauss-Seidel algorithms to enable parallel implementation.
However, this reformulation substantially increases the problem dimension in terms of both the number of variables and constraints.
This can slow down convergence and we demonstrate this empirically when we compare it with our proposed method.

Recently, \cite{liu2019linearized} suggested an alternative strategy to enable parallel computation.
Similar to the prox-linear method \cite{chen1994proximal}, it proposes to linearize around the previous iterate the quadratic penalty term in the augmented Lagrangian function, which makes the latter decomposable with respect to the $x$ blocks.
Although the method is shown to asymptotically converge for \eqref{eq:nlp} (no convergence rate type is provided), %
it can be slow to converge in practice for models with highly nonconvex constraints.
Indeed, that analysis suggests an augmented Lagrangian penalty parameter that scales as $O(\theta^4)$ posing numerical challenges.
Empirically, that method failed to converge for the smallest of our test instances with $T = 3$.

\subsection{Contributions}
We propose a novel distributed Jacobi scheme for solving the block-structured optimization problem \eqref{eq:nlp_general} with smooth nonconvex objective functions $f_t$ and constraint sets $X_t$ that are algebraically described by smooth nonconvex functions $c_t$.
The scheme only requires local solutions of individual block NLP problems.
In contrast to solving the higher-dimensional reformulation \eqref{eq:nlp_variable_splitting}, our algorithm directly performs Jacobi updates of the augmented Lagrangian function with respect to the linear coupling constraints in \eqref{eq:nlp}, in which optimizing a single variable block does not require the newest values of the other blocks.
This makes it particularly suitable for parallel computation, where the cost of each iteration can be reduced by a factor of $T$ compared to existing Gauss-Seidel schemes. %

We show that the algorithm converges globally to an $\epsilon$-approximate stationary point of \eqref{eq:nlp_general} in no more than $O(\epsilon^{-4})$ iterations and locally converges to an approximate local minimizer at a sublinear rate under mild assumptions.
The proof uses an easy-to-compute Lyapunov function that does not require a (typically unknown) local minimizer.
We supplement the analysis by showing (empirically) that convergence fails if the proximal weights in the algorithm are not chosen appropriately.
The algorithm can be interpreted as a nonconvex constrained extension of the Jacobi algorithms proposed in \cite{banjac2017jacobiConvex,chatzipanagiotis2017convergence,deng2017parallel,melo2017iteration}, and as a Jacobi extension of the Gauss-Seidel algorithm for nonconvex problems proposed in \cite{jiang2019structured}.
In contrast to the majority of existing approaches for nonconvex problems,
we provide a practical and automatic parameter tuning scheme, an open-source Julia implementation (that can be downloaded from \url{https://github.com/exanauts/ProxAL.jl}), and an empirical demonstration of convergence on a large-scale multi-period optimal power flow test case.

\subsection{Notation, Assumptions and Preliminaries}
For any integer $T$, we use $[T]$ to denote the index set $\{1, 2, \ldots, T\}$.
We use $x$ (without subscript) as shorthand for the entire vector of decisions $(x_1, \ldots, x_T) \in \mathbb{R}^n$, and $X$ as the corresponding shorthand for the feasible set $X_1 \times \ldots \times X_T$, where we define $n \coloneqq \sum_{t = 1}^T n_t$.
For an arbitrary vector $w = (w_1, \ldots, w_T) \in \mathbb{R}^{n}$, we define
$
Aw \coloneqq \sum_{t=1}^T A_t w_t
$,
and for any $t \in [T]$, we define
$
A_{\neq t}w_{\neq t} \coloneqq \sum_{s \in [T]\setminus\{t\}} A_{s} w_{s}
$.
We let $D \coloneqq \mathrm{diag}(A_1, \ldots, A_T) \in \mathbb{R}^{Tm \times n}$ denote the block-diagonal matrix with $A_1$, $\ldots$, $A_T$ along its diagonal.
Iterates at the $k^\text{th}$ iteration are denoted with superscript $k$.
The general normal cone~\cite[Definition 6.3]{RockafellarWets1998} to a set $X$ at a point $x \in X$ is denoted as $N_{X}(x)$.
For a vector $z$ and matrix $M$, we use $\norm{z}$ and $\norm{M}$ to denote their Euclidean and spectral norms, respectively.
When $M$ is positive (semi-)definite, we use $\norm{z}_M$ to denote the (semi-)norm $\sqrt{z^\top M z}$.
For a vector $z \in \mathbb{R}^N$ and set $S \subseteq \mathbb{R}^N$, we define $\dist{z}{S} \coloneqq \min_{w \in S} \norm{z - w}$.
For any $\epsilon > 0$, we let $B_\epsilon(z)$ denote the open Euclidean ball in $\mathbb{R}^N$ with center $z$ and radius $\epsilon$.
We use $\mathrm{I}$ and $\mathrm{0}$ to denote the identity and zero matrices, respectively; unless indicated otherwise, their dimensions should be clear from the context.
For two square symmetric matrices $M_1, M_2$, we use $M_1 \succ M_2$ ($M_1 \succeq M_2$) or $M_2 \prec M_1$ ($M_2 \preceq M_2$) to indicate that $M_1 - M_2$ is positive definite (positive semidefinite).

Throughout the paper, we make the following assumptions.
\begin{enumerate}[label=(A\arabic*),leftmargin=*]
    \item\label{assume:nonconvex_sets} $X_t$ %
    is non-empty and compact for all $t \in [T]$.
    \item\label{assume:nonconvex_functions} $f_t: \mathbb{R}^{n_t} \mapsto \mathbb{R}$ %
    is $C^2$ for all $t \in [T]$.
    \item\label{assume:full_rank_coupling} The matrix $A \coloneqq [A_1 \ldots A_T] \in \mathbb{R}^{m \times n}$ has full row rank.
    \item\label{assume:wellposed} Problem \eqref{eq:nlp_general} has a feasible solution.
\end{enumerate}

We say $x^* \in X$ is a stationary point of \eqref{eq:nlp_general}, if there exist Lagrange multipliers $\lambda^* \in \mathbb{R}^m$ satisfying:%
\begin{subequations}\label{eq:kkt_nlp_general}
    \begin{align}
        & %
        A x^* = b, \label{eq:kkt_nlp_general:coupling}\\
        & \nabla f_t(x_t^*) + A_t^\top \lambda^* \in -N_{X_t}(x_t^*), \;\; t \in [T]. \label{eq:kkt_nlp_general:stationarity_x}
    \end{align}
\end{subequations}
By introducing the primal and dual residual functions, $\pi : X \mapsto \mathbb{R}$ and $\delta_t : X_t \times \mathbb{R}^m \mapsto \mathbb{R}$, respectively, conditions \eqref{eq:kkt_nlp_general} can be stated as $\pi(x^*) = 0$ and $\delta_t(x_t^*, \lambda^*) = 0$, $t \in [T]$, where%
\begin{subequations}\label{eq:residual}
\begin{align}
    \pi(x) &= \norm{%
        Ax - b},
    \label{eq:residual_primal} \\
    \delta_t(x_t, \lambda) &= \dist{\nabla f_t(x_t) + A_t^\top \lambda}{-N_{X_t}(x_t)}, t \in [T] \label{eq:residual_dual}
\end{align}
\end{subequations}
Under an appropriate constraint qualification, the above are equivalent to the Karush-Kuhn-Tucker (KKT) conditions, and they must be necessarily satisfied if $x^*$ is a local solution of \eqref{eq:nlp_general}; see \cite[Lemma~12.9]{nocedal_wright_2006} for a discussion of the linear independence constraint qualification (LICQ).

The augmented Lagrangian function of \eqref{eq:nlp} with respect to its coupling constraints is parameterized with penalty parameters $\rho, \theta > 0$ and is defined as follows (with domain $X \times \mathbb{R}^m \times \mathbb{R}^m$):
\begin{equation}\label{eq:auglag_nlp}
    \begin{aligned}
        \mathcal{L}(x, z, \lambda) &= \sum_{t = 1}^T f_t(x_t) + \frac{\theta}{2} \norm{z}^2 \\
        &\phantom{=} + \lambda^\top \left[ %
        Ax + z - b\right] + \frac{\rho}{2} \norm{%
            Ax + z - b}^2.
    \end{aligned}
\end{equation}
For convenience, we let $\mathcal{L}(x_t; \bar{x}_{\neq t}, \bar{z}, \bar{\lambda})$ denote the above function with domain $X_t$ that is obtained by fixing all $x_{s}$ variables to $\bar{x}_{s}$, $s \in [T] \setminus \{t\}$, $z$ to $\bar{z}$, $\lambda$ to $\bar{\lambda}$, and ignoring any constants:
\begin{equation}\label{eq:auglag_nlp_fixed_xt}
    \begin{aligned}
        \mathcal{L}(x_t; \bar{x}_{\neq t}, \bar{z}, \bar{\lambda}) &= f_t(x_t) + \bar{\lambda}^\top A_t x_t \\
        &\phantom{=} + \frac{\rho}{2} \norm{A_t x_t + %
        A_{\neq t} \bar{x}_{\neq t} + \bar{z} - b}^2.
    \end{aligned}
\end{equation}

\section{Algorithm}
Algorithm~\ref{algo:main} outlines the basic distributed scheme for solving formulation \eqref{eq:nlp_general}.
Here, $x^0, z^0$ and $\lambda^0$ denote the initial guesses of the primal and dual variables of the penalty formulation \eqref{eq:nlp}, whereas $\rho$, $\theta$ and $\tau_x, \tau_z$ are scalars denoting the penalty parameters in \eqref{eq:nlp} and its augmented Lagrangian function, and proximal weights corresponding to the $x$ and $z$ variables, respectively.

\begin{algorithm}[!htbp]
    \caption{Basic distributed proximal Jacobi scheme}
    \label{algo:main}
    \begin{algorithmic}[1]
        \renewcommand{\algorithmicrequire}{\textbf{Input:}}
        \renewcommand{\algorithmicensure}{\textbf{Output:}}
        \REQUIRE $x^0 \in X$, $z^0 \in \mathbb{R}^m$, $\lambda^0 \in \mathbb{R}^m$, scalars $\rho, \theta, \tau_x, \tau_z > 0$.
        \FOR {$k = 1, 2, \ldots$}
        \STATE Update $x$: compute a local minimizer $x_t^{k}$ of \eqref{eq:x-update} by warm-starting with $x_t^{k-1}$ and solve in parallel for $t \in [T]$:
        \begin{equation}\label{eq:x-update}
            \min_{x_t \in X_t} \mathcal{L}(x_t; x_{\neq t}^{k-1}, z^{k-1}, \lambda^{k-1}) + \frac{\tau_x}{2} \norm{x_t - x_t^{k-1}}_{A_t^\top A_t}^2
        \end{equation}
        \label{algo:main:x-update}
        \STATE Update $z$:
        \begin{equation}\label{eq:z-update}
            z^{k} = \frac{\tau_z z^{k-1} - \rho\left[Ax^{k} - b\right] - \lambda^{k-1}}{\tau_z + \rho + \theta}
        \end{equation}
        \label{algo:main:z-update}
        \STATE Update $\lambda$:
        \begin{equation}\label{eq:lambda-update}
            \lambda^{k} = \lambda^{k-1} + \rho \left[Ax^{k} + z^{k} - b\right]
        \end{equation}
        \label{algo:main:lambda-update}
        \ENDFOR
    \end{algorithmic}
\end{algorithm}

We note a few important points about Algorithm~\ref{algo:main}.
The optimization problems \eqref{eq:x-update} in line~\ref{algo:main:x-update} that compute new values of the $x$ variables can be solved completely in parallel using any local NLP solver.
These problems minimize the sum of the augmented Lagrangian function \eqref{eq:auglag_nlp} and the proximal terms, over the $x_t$ variables alone.
Similarly, the update step \eqref{eq:z-update} is the closed-form solution of the quadratic problem that is defined by minimizing the sum of the augmented Lagrangian function \eqref{eq:auglag_nlp} and the proximal terms, over the $z$ variables alone, and by fixing the $x$ and $\lambda$ variables to $x^{k}$ and $\lambda^{k-1}$, respectively.
Even though this update occurs sequentially after solving \eqref{eq:x-update}, our analysis and results remain unchanged if we use a variant of the algorithm where the former is performed in parallel using only information about $x^{k-1}$.
We omit this variant for ease of exposition.
Indeed, since the update step \eqref{eq:z-update} is relatively cheap compared to solving problem \eqref{eq:x-update}, we can compute the former along with the latter locally on each node of a parallel computing architecture.
Finally, the dual variables in \eqref{eq:lambda-update} are updated as per the standard augmented Lagrangian method.

\subsection{Global convergence}
We first establish that under an appropriate choice of the scalar parameters $\rho, \theta, \tau_x, \tau_z$, the sequence $\{x^k\}$ generated by Algorithm~\ref{algo:main} converges to a point satisfying the first-order stationarity conditions \eqref{eq:kkt_nlp_general}.
The key idea is to establish that the sum of the augmented Lagrangian function and proximal terms is a candidate Lyapunov function:
\begin{equation}\label{eq:lyapunov_function}
    \begin{aligned}
        \Phi(x, z, \lambda, \hat{x}, \hat{z}) = \mathcal{L}(x, z, \lambda)  &+ \frac{\tau_z}{4} \norm{z - \hat{z}}^2 \\
        &+ \sum_{t=1}^T \frac{\tau_x}{4} \norm{x_t - \hat{x}_t}^2_{A_t^\top A_t}.
    \end{aligned}
\end{equation}
Specifically, we show that the sequence $\{\Phi^k\}$, which consists of $\Phi$ evaluated at the iterates generated by the algorithm, decreases monotonically and is bounded from below by $\hat{\Phi}$:
\begin{align}
    \Phi^k &\coloneqq \Phi(x^k, z^k, \lambda^k, x^{k-1}, z^{k-1}), \; k \geq 1, \label{eq:lyapunov_sequence_definition} \\
    \hat{\Phi} &\coloneqq \min_{x \in X} \sum_{t = 1}^T f_t(x_t). %
    \label{eq:lyapunov_lower_bound}
\end{align}
In the remainder of the paper, for $k \geq 1$, we define $\Delta z^k \coloneqq z^k - z^{k-1}$ and similarly define $\Delta x^k$, $\Delta \lambda^k$, and $\Delta \Phi^k$.
Also, let $\Delta z^0 \coloneqq -\tau_z^{-1}(\lambda^0 + \theta z^0)$, $\Delta x^0 \coloneqq 0$ and $\Phi^0 \coloneqq \mathcal{L}(x^0, z^0, \lambda^0) + \frac{\tau_z}{4} \|\Delta z^0\|^2$.
Finally, we let $\pi^k \coloneqq \pi(x^k) $ and $\delta^k_t \coloneqq \delta_t(x_t^k, \lambda^k)$.

\begin{theorem}[Global convergence]\label{theorem:global_convergence}
    Suppose that assumptions \ref{assume:nonconvex_sets}--\ref{assume:wellposed} hold, and the sequence $\{(x^k, z^k, \lambda^k)\}$ is generated by Algorithm~\ref{algo:main} with parameters $\rho, \theta, \tau_x, \tau_z > 0$, such that
    \begin{equation}\label{eq:eta_definition}
        \eta_x \coloneqq \frac{\tau_x}{4} - \frac{(T-1)\rho}{2} > 0, \;
        \eta_z \coloneqq \frac{\tau_z}{4} - \frac{2(\theta + \tau_z)^2}{\rho} > 0.
    \end{equation}
    Then, for all $K \geq 1$, there exists $j \in [K]$ such that
    \begin{gather*}
        \pi^j \leq \sqrt{\frac{2\big(\Phi^1 - \hat{\Phi}\big)}{\theta} \left(1 + \frac{2(\theta + \tau_z)^2}{K\eta_z\rho} \right)} %
        \\
        \delta_t^j \leq (\rho + \tau_x) \norm{A_t} \sqrt{\frac{2(T+1)\big(\Phi^1 - \Phi^K\big)}{K\min\{\eta_x, \eta_z\}}}, \; t \in [T]. %
    \end{gather*}
    In particular, for any $\epsilon \in (0, 1)$, if we choose
    \begin{equation*}%
    \theta = \frac{1}{\epsilon^2}, \quad
    \rho = \frac{64}{\epsilon^2}, \quad
    \tau_x = \frac{256\,(T-1)}{\epsilon^2}, \quad
    \tau_z = \frac{2}{\epsilon^2}, \quad
    \end{equation*}
    then after $K = O(\epsilon^{-4})$ iterations, the iterates $\{(x^k, z^k, \lambda^k)\}$ converge to an $\epsilon$-stationary point of problem \eqref{eq:nlp_general} satisfying
    $\min_{j \in [K]} \max\{\pi^j, \max_{t \in [T]}\delta_t^j\} = O(\epsilon)$. %
\end{theorem}
\begin{proof}
    See Appendix~\ref{sec:global_convergence_proof}.
\end{proof}

\begin{remark}\label{rem:global_convergence}
    The iteration complexity can be improved to $O(\epsilon^{-2})$, whenever the original problem \eqref{eq:nlp_general} has the same form as problem \eqref{eq:nlp}; see also \cite{jiang2019structured,sun2019two,wang2019global} for related discussions.
    Specifically, we can avoid introducing variables $z$ with penalty weights $O(\epsilon^{-2})$, whenever formulation~\eqref{eq:nlp_general} contains a variable block, say $x_T$, that is constrained only via the linear coupling constraints; i.e., $X_T = \mathbb{R}^{n_T}$, and its coupling matrix satisfies $A_T = \mathrm{I}$ (or more generally, its image $\Image{A_T} \supseteq \Image{[b, A_1, \ldots, A_{T-1}]}$), and its objective satisfies $\nabla^2 f_T (x_T) \preceq M\mathrm{I}$ for all $x \in X$, for some fixed $M > 0$. %
\end{remark}

\subsection{Local convergence}
We now establish that the sequence $\{x^k\}$ generated by Algorithm~\ref{algo:main} converges to an approximate local minimizer of the original problem \eqref{eq:nlp_general}, under some additional assumptions.
In contrast to the previous section, where we showed global convergence to stationary points of the original problem \eqref{eq:nlp_general}, the proof proceeds by showing convergence to local minimizers of the quadratic penalty formulation \eqref{eq:nlp}.
To relate local minimizers of the former to those of the latter, we introduce the (full) Lagrangian functions of problems \eqref{eq:nlp_general} and \eqref{eq:nlp}, $\Lambda$ and $\Lambda_\theta$, respectively.
Under assumption \ref{assume:equality_constraints} stated below, these functions have domains $\mathbb{R}^n \times \mathbb{R}^{r} \times \mathbb{R}^m$ and $\mathbb{R}^n \times \mathbb{R}^{r} \times \mathbb{R}^m \times \mathbb{R}^m$, respectively, where $r \coloneqq \sum_{t=1}^T r_t$, and are given by:
\begin{gather}
    \Lambda(x, \mu, \lambda) = \sum_{t = 1}^T [f_t(x_t) + \mu_t^\top c_t(x_t)] + \lambda^\top [Ax - b] \\
    \Lambda_\theta(x, z, \mu, \lambda) = \Lambda(x, \mu, \lambda) + \frac{\theta}{2} \norm{z}^2 + \lambda^\top z
\end{gather}
where $\mu_t \in \mathbb{R}^{r_t}$ and $\mu$ denotes the vector $(\mu_1, \ldots, \mu_T) \in \mathbb{R}^{r}$.
We make the following additional assumptions.
\begin{enumerate}[label=(A\arabic*),leftmargin=*]\setcounter{enumi}{4}
    \item\label{assume:equality_constraints} $X_t = \{x_t \in \mathbb{R}^{n_t}: c_t(x_t) = 0 \}$ can be algebraically described using only equality constraints, where $c_t :\mathbb{R}^{n_t} \mapsto \mathbb{R}^{r_t}$ 
    is $C^2$ for all $t \in [T]$.
    \item\label{assume:licq} For all $\theta > 0$,\footnote{In what follows, we suppress dependence of $x^*, z^*, \mu^*, \lambda^*$ on $\theta$ for ease of exposition.} there exist $x^* \in X$ and $z^* \in \mathbb{R}^m$, such that $(x^*, z^*)$ is a local solution of problem \eqref{eq:nlp}. Also, the constraint Jacobian $J(x^*)$ of the original problem \eqref{eq:nlp_general} at $x^*$
    has full row rank, where we define:
    \[\small
    J(x) =
    \begin{bmatrix}
        \nabla c(x)^\top \\ A
    \end{bmatrix}, 
    \nabla c(x) = 
    \begin{bmatrix}
        \nabla c_1(x_1) & \ldots & 0 \\
        \vdots          & \ddots & \vdots \\
        0               & \ldots & \nabla c_T(x_T)
    \end{bmatrix}.
    \]
    \item\label{assume:sosc} For all $\theta > 0$, there exist $\mu^* \in \mathbb{R}^r$ and $\lambda^* \in \mathbb{R}^m$,
    such that for the same $x^*, z^*$ in assumption \ref{assume:licq}, we have
    \begin{gather*}
        w_x^\top \nabla^2_{xx} \Lambda_\theta(x^*, z^*, \mu^*, \lambda^*) w_x + \theta w_z^\top w_z > 0,
    \end{gather*}
    for all $(w_x, w_z) \in \mathbb{R}^n \times \mathbb{R}^m$ satisfying $\nabla c(x^*)^\top w_x = 0$, $A w_x + w_z = 0$ and $\norm{w_x} + \norm{w_z} > 0$.
    \item\label{assume:licq:subproblem}
    For all $k \geq 1$, the constraint Jacobian of problem~\eqref{eq:x-update} at its optimal solution, $\nabla c_t(x_t^k)^\top$, has full row rank.
    \item\label{assume:subproblem:regular}
    For all $k \geq 1$, the $x$-update step in line~\ref{algo:main:x-update} computes a local minimizer of problem~\eqref{eq:x-update} that is closest to either $x_t^{k-1}$ if $A_t$ has full column rank or to $x_t^*$ otherwise.
\end{enumerate}
We make some remarks about these assumptions.
Assumption~\ref{assume:equality_constraints} is without loss of generality since inequalities can be reformulated as equality constraints by adding squares of additional slack variables, and it is also used in the convergence analyses of existing augmented Lagrangian methods (e.g., see~\cite{bertsekas2014constrained,sun2019two,harwood2021analysis}).
Assumptions~\ref{assume:licq} and~\ref{assume:sosc} state that the linear independence constraint qualification (LICQ) and second-order sufficient conditions (SOSC) are satisfied, respectively, at the local solution $(x^*, z^*)$ of the quadratic penalty formulation \eqref{eq:nlp} for all (sufficiently large) $\theta > 0$.
Assumption~\ref{assume:licq:subproblem} states that LICQ is also satisfied at local solutions of problem~\eqref{eq:x-update} that correspond to the $x$-updates.
As before, these are fairly standard to establish local convergence of augmented Lagrangian methods.
Finally, Assumption~\ref{assume:subproblem:regular} is satisfied by any well-behaved NLP solver if $A_t$ has full column rank since all variables have an associated proximal term in problem~\eqref{eq:x-update} in this case; when $A_t$ does not have full column rank, the assumption is necessary for convergence and it is implicit in classical analyses of the standard augmented Lagrangian method \cite[see eq.(6) in Proposition~2.4]{bertsekas2014constrained}.
\begin{remark}
    Assumption~\ref{assume:subproblem:regular} can be relaxed if the proximal term in problem~\eqref{eq:x-update} is replaced with $\frac{\tau_x}{2}\|x - x^{k-1}\|^2_{P_t}$, where $P_t \succeq A_t^\top A_t$ has full column rank. It can be shown that local (and global) convergence continues to hold in this case (with minor adjustments to the statement of Theorem~\ref{theorem:local_convergence}) if we relax \ref{assume:subproblem:regular} to require only that problem~\eqref{eq:x-update} computes a local minimizer that is closest to $x_t^{k-1}$. We do not present this generalization for ease of exposition.
\end{remark}

The key idea is to show that the iterates $(x^k, z^k)$ generated by Algorithm~\ref{algo:main}
correspond to local minimizers of a perturbed variant of problem~\eqref{eq:nlp},
that is parameterized by $p \in \mathbb{R}^m$ and $d = (d_1, \ldots, d_T, d_z) \in \mathbb{R}^{n+m}$.
These parameters correspond to the primal and dual residuals of the quadratic penalty formulation~\eqref{eq:nlp} at the iterates of Algorithm~\ref{algo:main}.
Specifically, for any $k \geq 1$, we define these as follows:
\begin{subequations}\label{eq:perturbation}
    \begin{align}
        p^k &\coloneqq Ax^k + z^k - b \label{eq:perturbation:primal} \\
        d_t^k &\coloneqq \rho A_t^\top A_{\neq t} \Delta x_{\neq t}^k - \rho A_t^\top \Delta z^k - \tau_x A_t^\top A_t \Delta x_t^k \label{eq:perturbation:dual_xt} \\
        d_z^k &\coloneqq -\tau_z \Delta z^k \label{eq:perturbation:dual_z}
    \end{align}
\end{subequations}
In the following theorem, recall that $D = \mathop{\mathrm{diag}}(A_1, \ldots, A_T)$ so that $Dx = (A_1 x_1, \ldots, A_T x_T) \in \mathbb{R}^{Tm}$.
\begin{theorem}[Local convergence]\label{theorem:local_convergence}
    Suppose that assumptions \ref{assume:nonconvex_sets}--\ref{assume:subproblem:regular} hold, $\epsilon \in (0, 1)$ is a fixed constant,
    and parameters $\rho, \theta, \tau_x, \tau_z > 0$ are chosen such that $\theta = \epsilon^{-2}$ and \eqref{eq:eta_definition} holds.
    Then, there exist constants $\bar{\rho}, \epsilon' > 0$ such that
    if
    $\rho > \bar{\rho}$ and
    the sequence $\{(x^k, z^k, \lambda^k)\}$ is generated by Algorithm~\ref{algo:main}
    with input $(x^0, z^0, \lambda^0)$ satisfying
    $\norm{(Dx^0, z^0, \lambda^0) - (Dx^*, z^*, \lambda^*)} < \epsilon'$
    and
    $\lambda^0 = -\theta z^0$,
    then the sequence $\{x^k\}$ converges at a sublinear rate to the $\epsilon$-approximate local minimizer $x^*$ of problem~\eqref{eq:nlp_general} satisfying:
    \begin{gather*}
        \max\left\{\pi(x^*), \max_{t \in [T]} \delta_t(x_t^*, \lambda^*)\right\} = O(\epsilon), \\
        w^\top \nabla^2_{xx} \Lambda(x^*, \mu^*, \lambda^*) w > 0, \; \forall w: J(x^*) w = 0, w \neq 0.
    \end{gather*}
\end{theorem}
\begin{proof}
    See Appendix~\ref{sec:local_convergence_proof}.
\end{proof}

\subsection{Adaptive parameter tuning}
In practice, the values for the parameters $\rho, \theta, \tau_x, \tau_z$ that are suggested in Theorems~\ref{theorem:global_convergence} and~\ref{theorem:local_convergence} are quite conservative and can be significantly larger than what is required for convergence.
Therefore, Algorithm~\ref{algo:tuning} presents a practical strategy for adaptively tuning these parameter values.

\begin{algorithm}[!htbp]
    \caption{Adaptive distributed proximal Jacobi scheme}
    \label{algo:tuning}
    \begin{algorithmic}[1]
        \renewcommand{\algorithmicrequire}{\textbf{Input:}}
        \renewcommand{\algorithmicensure}{\textbf{Output:}}
        \REQUIRE $\epsilon \in (0, 1)$, $\rho_0, \omega, \kappa_x, \kappa_z, \zeta, \Psi > 0$, $\nu_x, \nu_\rho, \nu_\theta, \chi > 1$
        \STATE Initialize
        $\theta = \epsilon^{-2}$,
        $\rho = \rho_0$, %
        $\tau_x = \kappa_x \rho$,
        $\tau_z = \kappa_z \rho$,
        $\psi = 0$.
        \FOR {$k = 1, 2, \ldots$}
        \STATE Update $x$, $z$, $\lambda$ as per Algorithm~\ref{algo:main}
        \IF {$\Phi^k - \Phi^{k-1} > \zeta | \Phi^k |$}
        \STATE $\tau_x \gets \min\{\nu_x \tau_x, (2T-1) \rho\}$ \label{algo:tuning:taux-update}
        \ENDIF
        \IF {$\max\left\{\|{p^k}\|_\infty, \|{d^k}\|_\infty \right\} \leq \epsilon$ and $\|{Ax^k - b}\|_\infty > \epsilon$}
        \STATE $\theta \gets \nu_\theta \theta$ \label{algo:tuning:theta-update}
        \ENDIF
        \IF {$\|{p^k}\|_\infty > \chi \|{d^k}\|_\infty$ and $\rho < \omega \theta$}
        \STATE
        $\rho \gets \min\{\nu_\rho \rho, \omega \theta\}$,
        $\tau_x \gets \kappa_x \rho$,
        $\tau_z \gets \kappa_z \rho$
        \ELSIF {$\|{d^k}\|_\infty > \chi \|{p^k}\|_\infty$ and $\psi < \Psi$}
        \STATE
        $\rho \gets \rho/\nu_\rho$,
        $\tau_x \gets \kappa_x \rho$,
        $\tau_z \gets \kappa_z \rho$,
        $\psi \gets \psi + 1$
        \ENDIF
        \IIf {$\|{Ax^k - b}\|_\infty \leq \epsilon$} \textbf{stop} \label{algo:tuning:stop-condition}
        \EndIIf
        \ENDFOR
    \end{algorithmic}
\end{algorithm}

The parameter tuning rules are motivated from the proofs of Theorems~\ref{theorem:global_convergence} and~\ref{theorem:local_convergence}.
In particular, the (cheaply computable) difference in Lyapunov function values, $\Phi^k - \Phi^{k-1}$, is guaranteed to be negative, whenever the penalty parameters $\rho, \theta$ and proximal weights $\tau_x, \tau_z$ are large enough to ensure $\eta_x, \eta_z > 0$.
Therefore, Algorithm~\ref{algo:tuning} starts with relatively small values of these parameters and gradually adjusts them.
In particular, $\tau_x$ is increased in line~\ref{algo:tuning:taux-update} by a factor $\nu_x > 1$ whenever $\Delta \Phi^k$ (appropriately scaled by $|\Phi^k|$) is larger than some $\zeta > 0$; however, the increase is limited to $(2T-1)\rho$ based on the definition \eqref{eq:eta_definition} of $\eta_x$.
The penalty parameter $\theta$ is increased by a factor $\nu_\theta > 1$ in line~\ref{algo:tuning:theta-update} whenever the primal and dual residuals with respect to quadratic penalty formulation~\eqref{eq:nlp}, $p^k$ and $d^k$, are smaller than the tolerance $\epsilon$, but the true primal residual $Ax^k - b$ continues to be larger than $\epsilon$.

The parameter $\rho$ is increased or decreased by a factor $\nu_\rho > 1$ to ensure that the magnitudes of $p^k$ and $d^k$ are within a factor $\chi > 1$ of each other, similar to~\cite{boyd2011distributed}.
However, the number of times $\rho$ is allowed to decrease is bounded by some integer $\Psi > 0$, and moreover, it is not allowed to increase after $\rho = \omega \theta$.
Therefore, $\rho$ will not be adjusted an infinite number of times.
The proximal weight $\tau_z$ is fixed at $\kappa_z \rho$, where $\kappa_z$ is selected to ensure that $\eta_z > 0$ whenever $\rho$ is also sufficiently large.
In particular, the proof of Theorem~\ref{theorem:global_convergence} shows that whenever $\rho/\theta = \omega > 32$, then $\kappa_z = 1/32$ suffices.
Finally, since Theorem~\ref{theorem:global_convergence} ensures that the termination condition in line~\ref{algo:tuning:stop-condition} will be met after a finite number of iterations (e.g., when $\eta_x, \eta_z > 0$), the parameter $\theta$ will not be increased an infinite number of times in line~\ref{algo:tuning:theta-update}.
Therefore, Algorithm~\ref{algo:tuning} is guaranteed to converge to an $\epsilon$-approximate stationary point of the original problem~\eqref{eq:nlp_general} after $O(\epsilon^{-4})$ iterations.

\section{Numerical Experiments}

We demonstrate the computational performance of the proximal Jacobi scheme on multi-period AC Optimal Power Flow (ACOPF) problems. %
The ACOPF is used to determine optimal dispatch levels of generators to satisfy electrical loads in power systems \cite{frank2016introduction}.
Its multi-period variant entails the solution of multiple single-period ACOPF problems that are coupled over a long time horizon.
Such extended horizons are essential to model operational constraints such as ramping limits of thermal generators, as well as planning decisions involving energy storage units \cite{maffei2014optimal} or production costing with accurate representations of system voltages \cite{castillo2016unit}.

We consider a $T$-period ACOPF problem, where
the variables $x_t = (p^g_{t}, q^g_{t}, V_t, \vartheta_t)$ in time period~$t$ are the real ($p^g_{t}$) and reactive ($q^g_{t}$) power generation levels and nodal voltage magnitudes ($V_t$) and angles ($\vartheta_t$), respectively.
The constraint set $X_t$ captures AC power balances with respect to the real ($p^d_t$) and reactive loads ($q^d_t$) in time period~$t$,
whereas the objective function $f_t(p^g_{t})$ minimizes the total production cost in that time period.
For simplicity, we ignore constraints on transmission line flow limits.
The single-period ACOPF problems are coupled via inter-temporal constraints that capture the physical ramping limits of generators; that is, generator~$i$ cannot change its real power output between periods $t$ and $t+1$ by more than $r_i \Delta_t$ units, where $\Delta_t$ denotes the length of period~$t$.
Formally, if $G$ and $B$ denote the sets of generators and buses, respectively,
then this problem can be formulated as follows.
\begin{equation*}\label{eq:acopf}
    \begin{aligned}
        \displaystyle\mathop{\text{minimize}}_{x_1, \ldots, x_T} \;\; & \displaystyle \sum_{t \in [T]} f_t(p^g_{t}) \\
        \text{subject to} \;\;& \displaystyle x_t = (p^g_{t}, q^g_{t}, V_t, \vartheta_t) \in [\underline{x}, \overline{x}], \; t \in [T], \\
        & \sum_{j \in G_i} p^g_{jt} - p^d_{it} = c_i^\mathrm{re}(V_t, \vartheta_t), \; i \in B, \; t \in [T], \\
        & \sum_{j \in G_i} q^g_{jt} - q^d_{it} = c_i^\mathrm{im}(V_t, \vartheta_t), \; i \in B, \; t \in [T], \\
        & \displaystyle \left\lvert p^g_{i,t+1} - p^g_{it} \right \rvert \leq r_{i} \Delta_t, \; i \in G, \; t \in [T-1],
    \end{aligned}
\end{equation*}
where $[\underline{x}, \overline{x}]$ represent variable bounds, $G_i$ is the set of generators connected to bus~$i$,
and $c_i^\mathrm{re}$ and $c_i^\mathrm{im}$ are smooth nonconvex functions of $V_t$ and $\vartheta_t$.
We denote by $Y = Y^\mathrm{re} + \sqrt{-1} Y^\mathrm{im}$ the $|B| \times |B|$ admittance matrix (with $Y_{ii} = y_{ii} - \sum_{j \neq i} Y_{ij}$, $y_{ii} = y^\mathrm{re}_{ii} + \sqrt{-1} y^\mathrm{im}_{ii}$,
where $y^\mathrm{re}_{ii}$ and $y^\mathrm{im}_{ii}$ are the shunt conductance and susceptance at bus $i \in B$, respectively).
Using the polar formulation, the functions $c_i^{\mathrm{re},\mathrm{im}}$ can be written, for all buses $i\in B$, as \cite{molzahn2019survey}
\begingroup
\small
\begin{equation*}
    \begin{aligned}
        c_i^\mathrm{re}(V, \vartheta) &= Y^\mathrm{re}_{ii} V_i^2 \\
        &+\sum_{j \in N_i}V_i V_j \big[Y^\mathrm{re}_{ij}\cos(\vartheta_{i} - \vartheta_{j}) + Y^\mathrm{im}_{ij}\sin(\vartheta_{i} - \vartheta_{j}) \big], \\
        c_i^\mathrm{im}(V, \vartheta) &= -Y^\mathrm{im}_{ii} V_i^2 \\
        &+\sum_{j \in N_i}V_i V_j \big[Y^\mathrm{re}_{ij}\sin(\vartheta_{i} - \vartheta_{j})- Y^\mathrm{im}_{ij}\cos(\vartheta_{i} - \vartheta_{j}) \big],
    \end{aligned}
\end{equation*}
\endgroup
where $N_i \subseteq B$ is the set of neighboring buses of $i$.

The multi-period ACOPF can be formulated as an instance of problem~\eqref{eq:nlp_general} by introducing extra variables $s^g_{i,t+1} \in [0, 2r_i\Delta_t]$ and reformulating the inter-temporal ramping limits as equality constraints:
$p^g_{i,t+1} - p^g_{it} + s^g_{i,t+1} =  r_{i} \Delta_t$.
Note that although the objective function is linear or convex quadratic, the nonconvexity of $c^\mathrm{re}$ and $c^\mathrm{im}$ makes the problem nonconvex.

We consider the standard IEEE test cases `118', `1354pegase', and `9241pegase' available from MATPOWER \cite{zimmerman2010matpower}.
Since the original data only provides a single vector of loads, we generate a load profile over multiple periods
as follows.
We first obtain hourly load data from New England ISO~\cite{nyiso} over a typical week (7~days) of operations.
We then use this $T = 24\times 7 = 168$ period profile as a multiplier for the load at each bus, downscaling if required to ensure single-period ACOPF feasibility.
Thus, the load distribution (across buses) in each period remains the same as the original MATPOWER case, and only the total load (across time periods) follows the imposed profile.
We consider fairly stringent ramp limits $r_i \in [0.33\%, 0.50\%] \times p^{\max}_{i}$ (\textit{per minute}, which is the standard unit for this application class), where $p^{\max}_{i}$ is the maximum rated output of generator~$i$ (note that $\Delta_t = 60$~min).
For each test case, the objective function is scaled by $10^{-3}$ (roughly $T^{-1}$) to ensure $\rho = O(1)$.
All runs are initialized with $x^0 = (\underline{x} + \overline{x})/2$, $z^0 = 0$, $\lambda^0 = 0$ and
unless mentioned otherwise, following parameter values are used in Algorithm~\ref{algo:tuning}:
$\rho_0 = 10^{-3}, \kappa_x = 2$ for case 118 and $\rho_0 = 10^{-5}, \kappa_x = 2.5$ otherwise,
and
$\omega = 32$,
$\kappa_z = 1/32$,
$\zeta = 10^{-4}$,
$\nu_x = 2$,
$\nu_\rho = 2$,
$\nu_\theta = 10$,
$\chi = 10$,
$\Psi = 100$.

Our algorithm is implemented in Julia and is available at \url{https://github.com/exanauts/ProxAL.jl}.
It uses the Message Passing Interface (MPI) which allows the use of distributed parallel computing resources.
All NLP subproblems were solved using the JuMP modeling interface~\cite{DunningHuchetteLubin2017} and Ipopt~\cite{wachter2006implementation} (with default options) as the NLP solver.
The computational times reported in Section~\ref{sec:results:scalability} were obtained on the Summit supercomputer at Oak Ridge National Laboratory~\cite{summit}.

\subsection{Role of proximal terms}\label{sec:results:prox_terms}

The proof of Theorem~\ref{theorem:global_convergence} shows that for fixed values of $\theta$ and $\rho$, the Lyapunov $\{\Phi^k\}$ must be monotonically decreasing, whenever the proximal weights $\tau_x$ are sufficiently large.
To verify this empirically, we fix $\theta = 10^6$, $\rho = 1$, $\tau_z = \kappa_z \rho$ and then consider $\tau_x \in \{0, 1, 2\}$; we also disable their automatic updates in Algorithm~\ref{algo:tuning}.
Figure~\ref{figure:lyapunov} shows the values of the corresponding Lyapunov sequences $\{\big(\Phi^k - \underline{\Phi}\big)/{\underline{\Phi}}\}$, where $\underline{\Phi}$ is a normalizing constant equal to the smallest observed value of $\Phi^k$ (across 100 iterations).
We find that when the proximal terms are absent ($\tau_x = 0$) or when they are present but not sufficiently large ($\tau_x = 1$), the Lyapunov sequence (and hence also $\{(x^k, z^k)\}$) diverges.
This is contrast to classical Gauss-Seidel based algorithms~\cite{sun2019two,jiang2019structured}, which are not as critically dependent on proximal terms.
Indeed, the Jacobi nature of the $x$-update in Algorithm~\ref{algo:main} requires that the new iterate $x^k$ be sufficiently close to its old value $x^{k-1}$ for convergence.

\begin{figure}[!t]
    \centering
    \includegraphics[width=\linewidth]{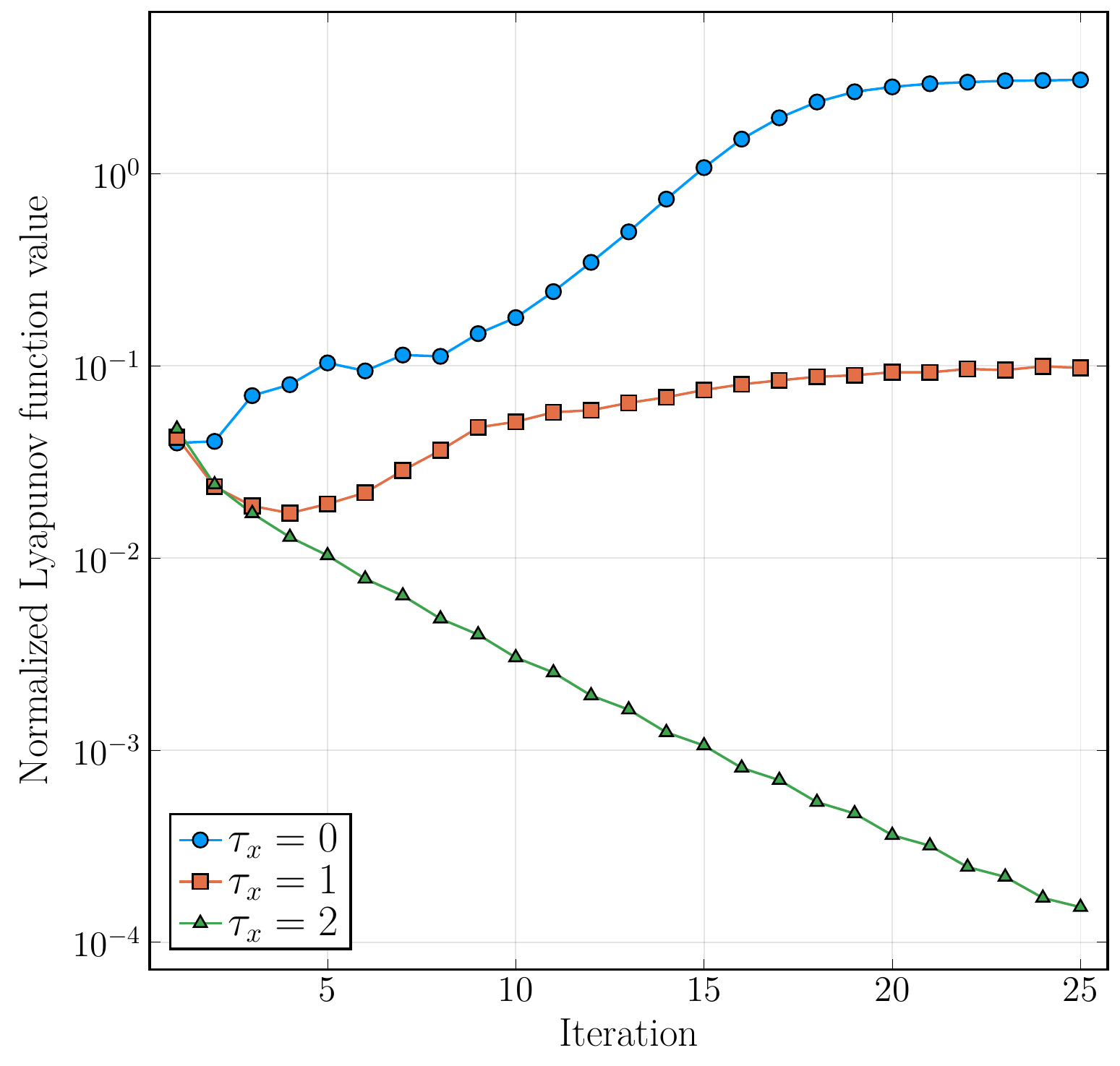}
    \caption{Lyapunov sequence for increasing values of the proximal weight $\tau_x$ and for fixed values of $\theta$, $\rho$ and $\tau_z$.}
    \label{figure:lyapunov}
\end{figure}

\subsection{Comparison with existing method}\label{sec:results:comparison}

We compare the performance of our proposed scheme with the two-level ADMM algorithm proposed in \cite{sun2019two}.
Since the latter uses Gauss-Seidel or sequential updating rules, which are not amenable to distributed parallelization, the two algorithms are not directly comparable.
Therefore, to solve the multi-period ACOPF with the two-level algorithm, we reformulate it using the variable splitting technique described in \cite[Section~1.2]{sun2019two}.
Specifically, for each $t \in [T]$, we introduce a so-called global copy $\bar{p}^g_t$ of $p^g_t$,
and for each $t \in [T-1]$, we introduce a so-called local copy $\hat{p}^{g}_{t+1}$ (local to block~$t+1$) of $p^g_{t}$.
The ramping limit for period $t \in [T-1]$ is then expressed as:
$p^g_{i,t+1} - \hat{p}^g_{i,t+1} + s^g_{i,t+1} =  r_{i} \Delta_t$,
which is added as an additional constraint to $X_{t+1}$.
Finally, the coupling constraints linking the various blocks can be expressed as:
$p^g_{t} = \bar{p}^g_t$ for $t \in [T]$,
and
$\hat{p}^g_{t+1} = \bar{p}^g_t$ for $t \in [T-1]$.
For brevity, we consider only the 118-bus case with a ramp limit of $0.33\%$.

As suggested in \cite[Section~6]{sun2019two}, our implementation of the two-level ADMM uses $\omega = 0.75$, $\gamma = 1.5$, $[\underline{\lambda}, \bar{\lambda}] = [-10^6, 10^6]$, $\beta^1 = 1000$, $\rho^k = c_1 \beta^k$, and the $k^\mathrm{th}$ level inner-level loop is terminated based on \cite[eq.14c]{sun2019two} with $\epsilon_3^k = \sqrt{2n_g(T-1)}/(c_2 k \rho^k)$, where $n_g$ is the number of generators and all other symbols refer to the notation in \cite{sun2019two}.
After trying $(c_1, c_2) \in \{2 \times 10^f : f = -6, -4, -2, 0, 2\} \times \{1, 1000\}$, we set
$c_1 = 2\times 10^{-4}$ and $c_2 = 1000$, since it produced the smallest infinity norm of the combined primal and dual residual vector \cite[eq.14a-14c]{sun2019two} after 50 cumulative iterations.
We set the tolerance $\epsilon = 10^{-8}$ in Algorithm~\ref{algo:tuning}.

Figure~\ref{figure:comparison} shows the performance of our proximal Jacobi scheme and the two-level algorithm.
All residuals correspond to infinity norms; in our scheme, the plotted primal and dual residuals are $\|Ax^k - b\|_{\infty}$ and $\|d^k\|_{\infty}$; in the two-level scheme, they are $\max\{\|p^{g,k} - \bar{p}^{g,k}\|_{\infty}, \|\hat{p}^{g,k} - \bar{p}^{g,k}\|_{\infty}\}$ and the infinity norm of \cite[eq.14a-14b]{sun2019two} (plotted as a function of the number of cumulative iterations), respectively.
We observe that for the first few iterations (up to $25$) and for small target tolerances (up to $10^{-3}$), the two algorithms have similar behavior in terms of reducing the optimality errors.
However, our proposed scheme is able to decrease the residuals at an almost linear rate even for tighter tolerances, whereas the decrease for the two-level scheme seems to become worse than linear. %
This is likely because Algorithm~\ref{algo:tuning} allows $\rho$ to decrease; indeed, we observe in Figure~\ref{figure:comparison} that the primal and dual residuals are roughly within a factor of $\chi = 10$ of each other at all iterations.
Finally, note that the run times of both schemes are expected to be similar, since the per-iteration cost in either scheme is dominated by the solution of (single-period) ACOPF NLP problems.
The actual run time is a function of the number of parallel processes available which we study in the next section.

\begin{figure}[!t]
    \centering
    \includegraphics[width=\linewidth]{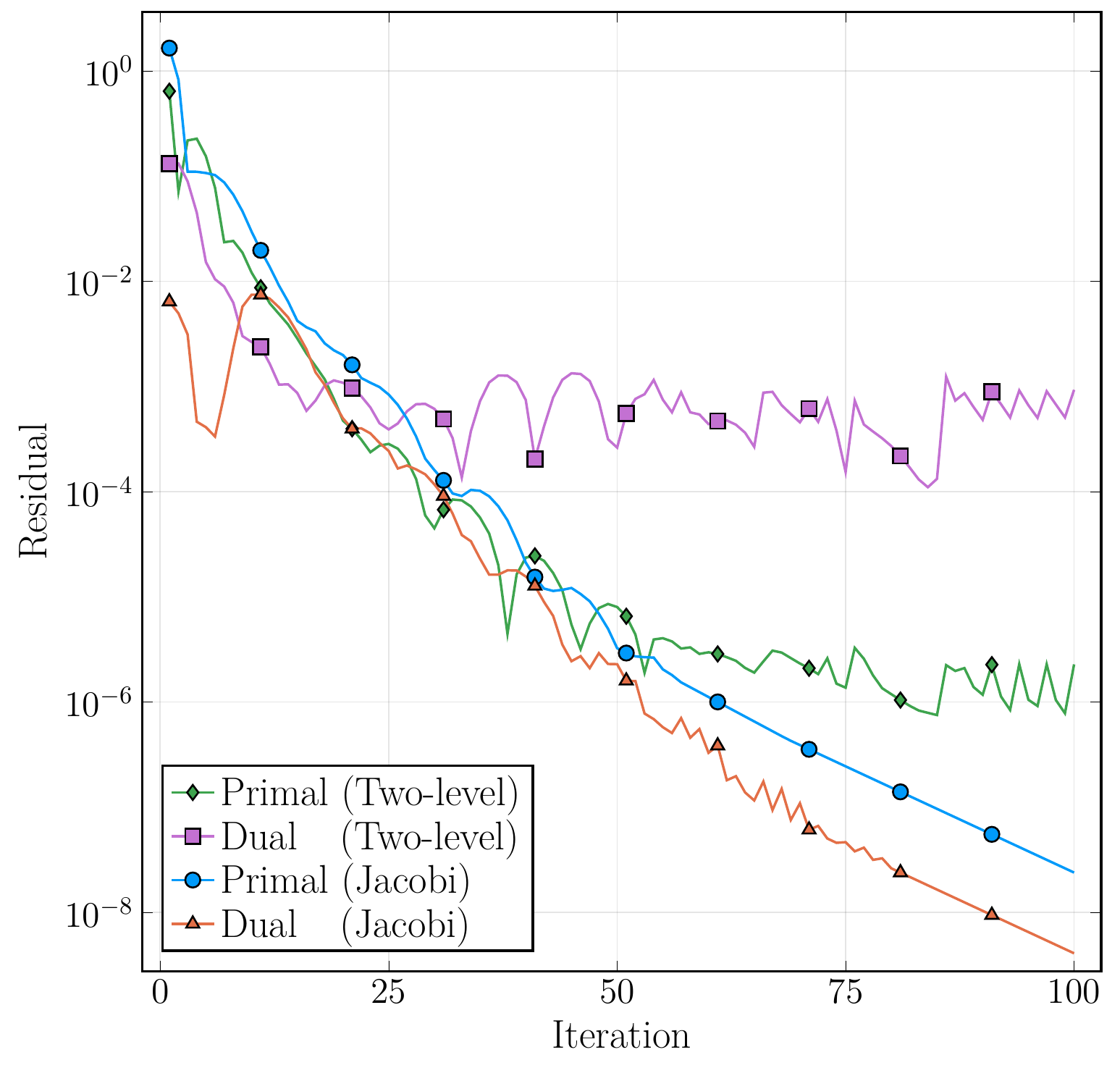}
    \caption{Computational comparison of the proposed proximal Jacobi scheme with the two-level algorithm proposed in \cite{sun2019two}.}
    \label{figure:comparison}
\end{figure}

\subsection{Scalability}\label{sec:results:scalability}

One of the main features of the proposed algorithm is its ability to make use of distributed parallel computing resources.
Therefore, we study its scalability with respect to both the problem size (for a fixed number of parallel processes) as well as the number of processes (for a fixed problem size).
Table~\ref{table:summary} summarizes the computational performance for the different test cases
using $T = 168$ parallel processes and a target tolerance of $\epsilon = 10^{-3}$ in each case.
We observe from Table~\ref{table:summary} that the number of iterations and run times increase with increasing network size and with decreasing values of the ramping limit (which makes the linear coupling constraints harder to satisfy).
Nevertheless, the total run time remains less than 100~minutes for the most difficult test case consisting of more than 3~million variables and constraints.
Also, Figure~\ref{figure:convergence} (which is plotted for the ramping limit of $0.33\%$)  shows that although the initial primal residuals are quite large, they are reduced by more than 4~orders of magnitude over the course of roughly 50 iterations.
Finally, Figure~\ref{figure:scaling} shows the empirically observed run times as a function of the number of parallel processes (MPI ranks) for the 118-bus test case with $r = 0.33\%$.
We find that our implementation exhibits a near-linear scaling in this particular instance.

\begin{table}[!thbp]
    \caption{Summary of computational performance using $T = 168$ parallel processes for convergence to $\epsilon = 10^{-3}$ tolerance.}
    \label{table:summary}
    \centering
    \begin{tabularx}{\columnwidth}{lrrccr}
        \toprule
        Case        & \# Vars.  & \# Cons.  & Ramp \% & \# Iters. & Time (s) \\
        \midrule
        \multirow{2}{*}{118}        & \multirow{2}{*}{66,810}     & \multirow{2}{*}{48,666}    &  0.33 &   24 &      5.2 \\ 
        &                             &                            &  0.50 &   13 &      3.9 \\ \midrule 
        \multirow{2}{*}{1354pegase} & \multirow{2}{*}{1,181,779}  & \multirow{2}{*}{696,259}   &  0.33 &   60 &    137.3 \\ 
        &                             &                            &  0.50 &   53 &    119.7 \\ \midrule 
        \multirow{2}{*}{9241pegase} & \multirow{2}{*}{3,235,756}  & \multirow{2}{*}{3,148,396} &  0.33 &   67 &  4,755.5 \\ 
        &                             &                            &  0.50 &   59 &  3,511.9 \\
        \bottomrule
    \end{tabularx}
\end{table}

\begin{figure}[!thbp]
    \centering
    \includegraphics[width=\linewidth]{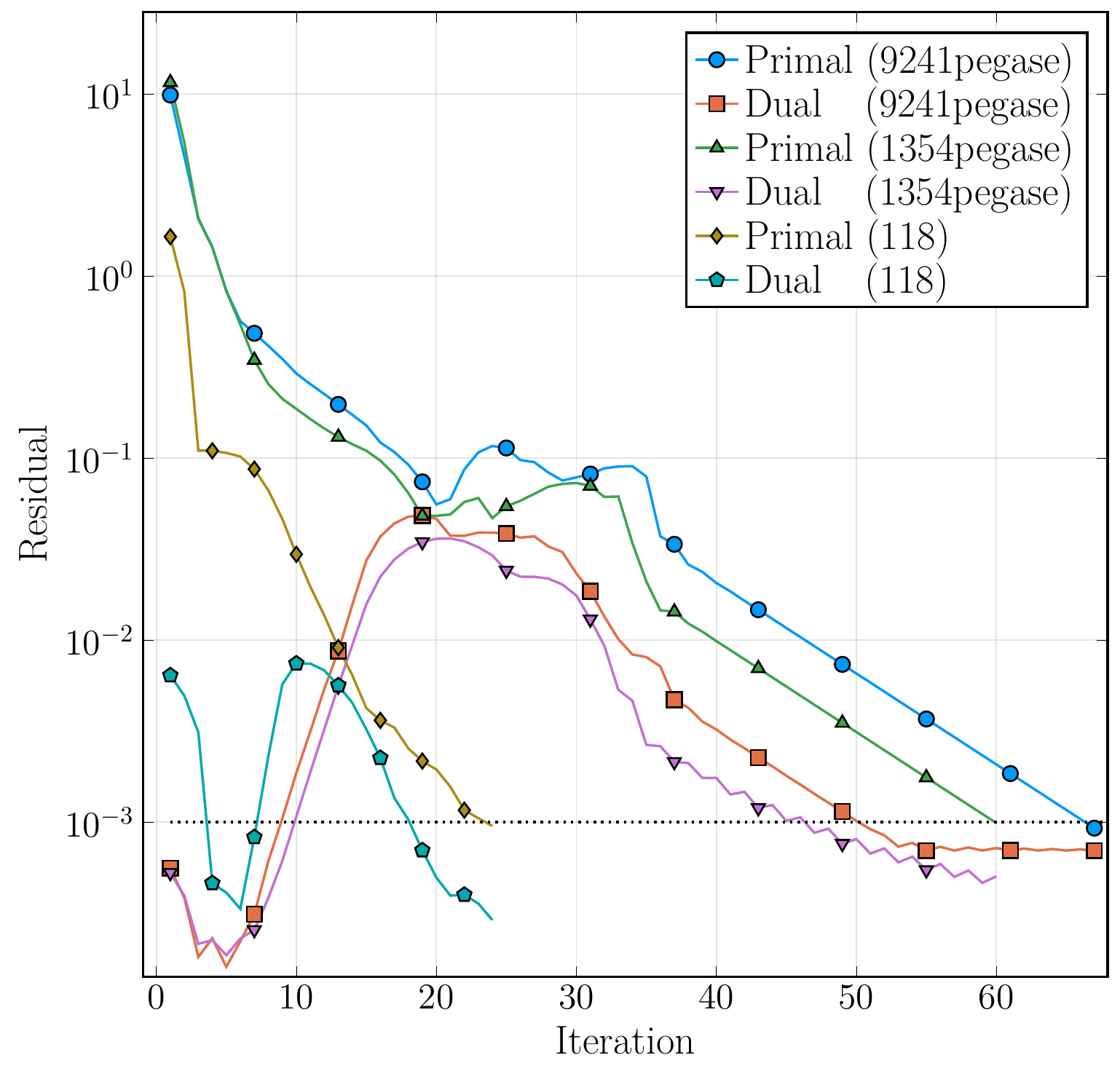}
    \caption{Convergence of the proposed scheme to a tolerance $\epsilon = 10^{-3}$.}
    \label{figure:convergence}
\end{figure}

\begin{figure}[!thbp]
    \centering
    \includegraphics[width=0.9\linewidth]{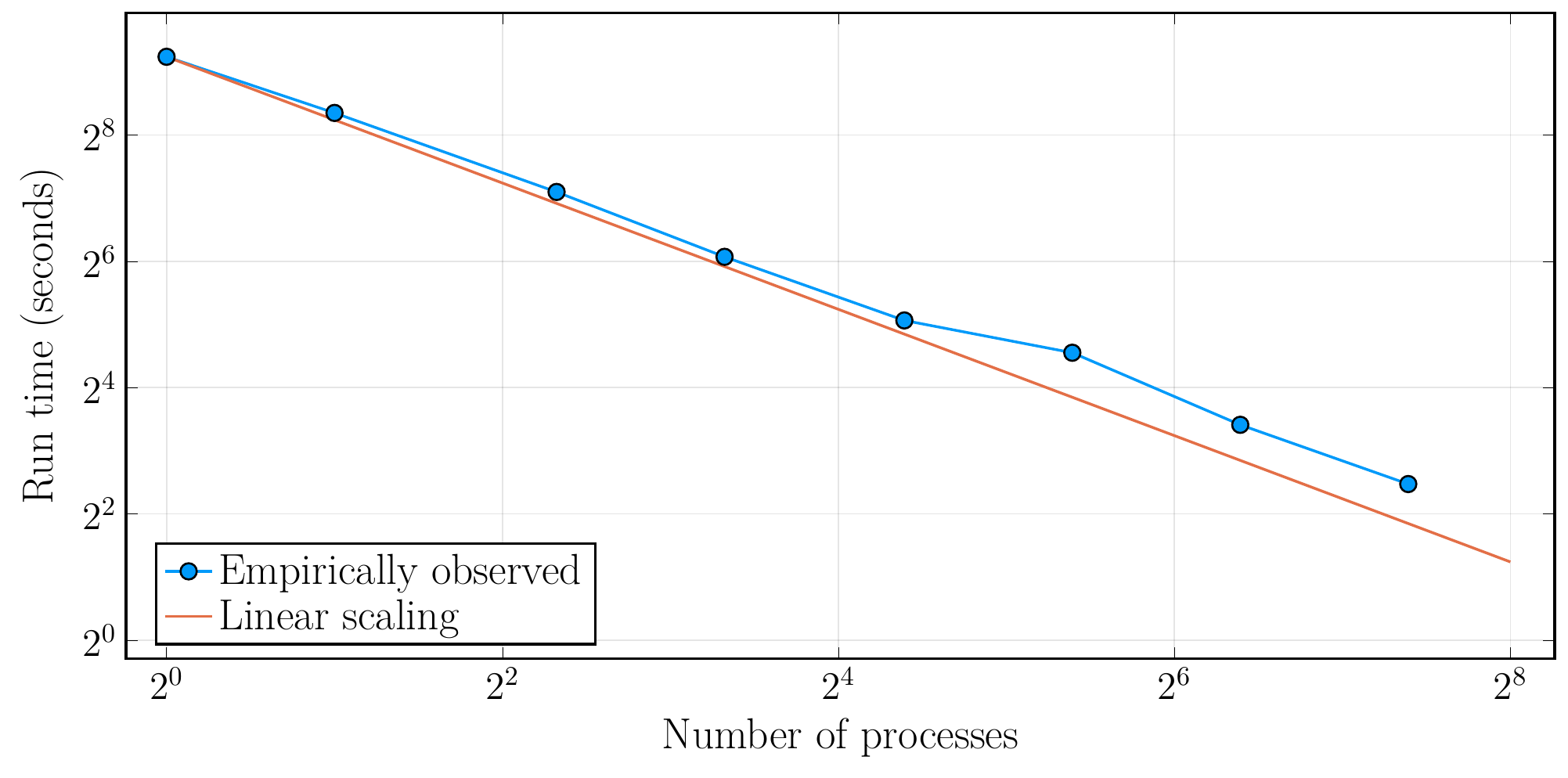}
    \caption{Scalability as a function of the number of parallel processes.}
    \label{figure:scaling}
\end{figure}

\section{Conclusions}

This paper proposed a distributed parallel decomposition algorithm for solving linearly coupled block-structured nonconvex constrained optimization problems.
Such structures naturally arise in several applications, including nonlinear model predictive control and stochastic optimization.
The algorithm performs Jacobi updates of the augmented Lagrangian function with appropriately chosen proximal terms, requiring only local solutions of the individual block NLP problems.
By constructing an easily computable Lyapunov function, we showed that the algorithm converges globally to $\epsilon$-approximate stationary points in $O(\epsilon^{-4})$ iterations as well as locally to $\epsilon$-approximate local minimizers of the original NLP at a sublinear rate.
We also provided a simple, practical and theoretically convergent variant of the algorithm where its parameters are adaptively tuned during the iterations.
Our numerical experiments show that it can outperform existing algorithms based on ADMM and that optimizing large-scale nonconvex AC optimal power flow problems over more than 100~time steps can be done in the order of roughly one~hour.
Future work includes the study of acceleration techniques such as inexact NLP solutions and momentum methods, integration with the two-level algorithm of~\cite{sun2019two}, as well as the incorporation of second-order information such as in~\cite{houska2016augmented}.

\appendices

\section{Global convergence proof}\label{sec:global_convergence_proof}
The proof of Theorem~\ref{theorem:global_convergence} requires several intermediate results.
Throughout, we suppose that the conditions in Theorem~\ref{theorem:global_convergence} hold; that is, assumptions \ref{assume:nonconvex_sets}--\ref{assume:wellposed} are satisfied and $\rho, \theta, \eta_x, \eta_z > 0$ are chosen such that \eqref{eq:eta_definition} is satisfied.

First, Lemma \ref{lemma:lyapunov_descent} establishes that the Lyapunov sequence $\{\Phi^k\}$ defined in \eqref{eq:lyapunov_sequence_definition} decreases by positive multiples of $\norm{\Delta x^k_t}^2_{A_t^\top A_t}$ and $\norm{\Delta z^k}^2$.
Second, Lemma \ref{lemma:lyapunov_lower_bound} shows that the sequence is bounded below by $\hat{\Phi}$ defined in \eqref{eq:lyapunov_lower_bound}.
Lemma \ref{lemma:delta_x_delta_z_bounds} uses these results to bound $\norm{\Delta x^k_t}^2_{A_t^\top A_t}$ and $\norm{\Delta z^k}^2$ as a function of the iteration index $k$.
Finally, Propositions \ref{prop:primal_residual_bound} and \ref{prop:dual_residual_bound} establish bounds on the primal and dual residuals, respectively.
The proof of Theorem~\ref{theorem:global_convergence} follows directly from these.
We use the following intermediate result to simplify exposition.
\begin{lemma}\label{lemma:lambda_z_relation}
    For all $k \geq 0$, we have:
    \begin{equation}\label{eq:lambda_z_relation}
        \lambda^k = -\theta z^k - \tau_z \Delta z^k.
    \end{equation}
\end{lemma}
\begin{proof}
    For $k = 0$, equation \eqref{eq:lambda_z_relation} is equivalent to the definition of $\Delta z^0$.
    For $k \geq 1$, the $\lambda$-update formula \eqref{eq:lambda-update} implies that
    $
    \lambda^{k-1} = \lambda^{k} - \rho \big[\sum_{t = 1}^T A_t x_t^{k} + z^{k} - b\big].
    $
    Substituting this in the $z$-update formula \eqref{eq:z-update} yields equation \eqref{eq:lambda_z_relation}.
\end{proof}
\begin{lemma}\label{lemma:lyapunov_descent}
    For all $k \geq 1$, we have:
    \begin{equation}\label{eq:delta_phi}
        \begin{aligned}
            \Delta \Phi^{k} \leq
            &-\eta_x \sum_{t = 1}^T \left(\norm{\Delta x_t^k}^2_{A_t^\top A_t} +  \norm{\Delta x_t^{k-1}}^2_{A_t^\top A_t}\right) \\
            &-\eta_z \left(\norm{\Delta z^k}^2 + \norm{\Delta z^{k-1}}^2\right).
        \end{aligned}
    \end{equation}
\end{lemma}
\begin{proof}
    Using the definitions \eqref{eq:lyapunov_function} and \eqref{eq:lyapunov_sequence_definition}, we have:
    \begin{equation}\label{eq:lemma_lyapunov_descent_delta_phi}
        \begin{aligned}
    \Delta \Phi^k &= \mathcal{L}(x^k, z^k, \lambda^k) - \mathcal{L}(x^{k-1}, z^{k-1}, \lambda^{k-1}) \\
    &\phantom{=}+ \sum_{t=1}^T \frac{\tau_x}{4} \left[\norm{\Delta x_t^k}_{A_t^\top A_t}^2 - \norm{\Delta x_t^{k-1}}_{A_t^\top A_t}^2 \right] \\
    &\phantom{=}+ \frac{\tau_z}{4} \left[\norm{\Delta z^k}^2 - \norm{\Delta z^{k-1}}^2 \right].
        \end{aligned}
    \end{equation}
The first term in the above expression can be written as:
    \begin{subequations}\label{eq:lemma_lyapunov_descent_ab}
        \begin{align}
        &\mathcal{L}(x^k, z^k, \lambda^k) - \mathcal{L}(x^{k-1}, z^{k-1}, \lambda^{k-1}) = (a) + (b)\\
        &(a) =
        \big[\mathcal{L}(x^k, z^k, \lambda^k) - \mathcal{L}(x^k, z^k, \lambda^{k-1})\big] \\
        &(b) =
        \big[\mathcal{L}(x^k, z^k, \lambda^{k-1}) - \mathcal{L}(x^{k-1}, z^{k-1}, \lambda^{k-1})\big]
        \end{align}
    \end{subequations}
First, we examine $(a)$. From the definitions of \eqref{eq:auglag_nlp} and \eqref{eq:lambda-update}, it follows that
$
(a) = \frac{1}{\rho}\norm{\Delta \lambda^k}^2.
$
Now, replacing $k$ with $k-1$ in \eqref{eq:lambda_z_relation}, we get
$
\lambda^{k-1} = -\theta z^{k-1} - \tau_z \Delta z^{k-1}
$
and subtracting this equation from \eqref{eq:lambda_z_relation}, we obtain:
\begin{equation}\label{eq:delta_lambda}
\Delta \lambda^k =
-(\theta + \tau_z) \Delta z^k + \tau_z \Delta z^{k-1}.
\end{equation}
Combining these and using the triangle inequality, we obtain:
\begin{equation}\label{eq:lemma_lyapunov_descent_a}
\begin{aligned}
    (a) &= \frac{1}{\rho}\norm{\Delta \lambda^k}^2 \\
    &\leq \frac{1}{\rho} \left((\theta + \tau_z) \norm{\Delta z^k} + \tau_z \norm{\Delta z^{k-1}} \right)^2 \\
    &\leq \frac{(\theta + \tau_z)^2}{\rho} \left(\norm{\Delta z^k} + \norm{\Delta z^{k-1}} \right)^2 \\
    &\leq \frac{2(\theta + \tau_z)^2}{\rho}  \left(\norm{\Delta z^k}^2 + \norm{\Delta z^{k-1}}^2\right), 
\end{aligned}
\end{equation}
where the second and third inequalities follow from $\theta > 0$ and $(a_1 + a_2)^2 \leq 2(a_1^2 + a_2^2)$ for arbitrary real $a_1, a_2$, respectively.

We now examine $(b)$. For arbitrary $t \in [T]$ and $j, k, l \geq 0$, let $(x_{<t}^j, x_t^k, x_{>t}^l)$ denote the vector obtained by stacking all $x_s^j$ for $s < t$, $x_s^k$ for $s = t$, and $x_s^l$ for $s > t$. The term $(b)$ can now be expressed as the following telescoping sum.
Note that since $\lambda$ is fixed at $\lambda^{k-1}$, we temporarily define $L(x, z) \coloneqq \mathcal{L}(x, z, \lambda^{k-1})$ to simplify exposition.
\begin{align}
    (b) &= \big[L(x^k, z^k) - L(x^k, z^{k-1})\big] \notag \\
    &\hspace{-1em}+\sum_{t=1}^T \big[ L(x_{<t}^k, x_t^k, x_{>t}^{k-1}, z^{k-1})  - L(x_{<t}^k, x_t^{k-1}, x_{>t}^{k-1}, z^{k-1}) \big] \notag \\
    &= \alpha + \sum_{t=1}^T (\beta_{t} + \gamma_t) \label{eq:lemma_lyapunov_descent_b}
\end{align}
where the expressions for $\beta_t$ and $\gamma_t$ can be verified by expanding out both sides of each summand:
\begin{align*}
    \alpha &= L(x^k, z^k) - L(x^k, z^{k-1}) \\
    & \leq -\frac{\tau_z}{2} \norm{\Delta z^k}^2 \\
    \beta_t &= L(x_{<t}^{k-1}, x_t^k, x_{>t}^{k-1}, z^{k-1}) - L(x_{<t}^{k-1}, x_t^{k-1}, x_{>t}^{k-1}, z^{k-1}) \\ %
    &\leq -\frac{\tau_x}{2} \norm{\Delta x^k_t}^2_{A_t^\top A_t} \\
    \gamma_t &= \left[L(x_{<t}^{k}, x_t^k, x_{>t}^{k-1}, z^{k-1}) - L(x_{<t}^{k-1}, x_t^k, x_{>t}^{k-1}, z^{k-1}) \right] \\
    &\phantom{=} -\left[L(x_{<t}^{k}, x_t^{k-1}, x_{>t}^{k-1}, z^{k-1}) -L(x_{<t}^{k-1}, x_t^{k-1}, x_{>t}^{k-1}, z^{k-1})\right] \\
    &= \rho \sum_{s=1}^{t-1} (A_t \Delta x_t^k)^\top (A_s \Delta x_s^k) \\
    &\leq \frac{\rho}{2} \sum_{s=1}^{t-1} \left(\norm{A_t \Delta x_t^k}^2 + \norm{A_s \Delta x_s^k}^2 \right) \\
    \sum_{t=1}^T \gamma_t &\leq \frac{\rho(T-1)}{2} \sum_{t=1}^T \norm{\Delta x_t^k}^2_{A_t^\top A_t} \\
    &\leq \frac{\rho(T-1)}{2} \sum_{t=1}^T \left(\norm{\Delta x_t^k}^2_{A_t^\top A_t} + \norm{\Delta x_t^{k-1}}^2_{A_t^\top A_t}\right).
\end{align*}
The $\alpha$-inequality %
is derived as follows:
substituting \eqref{eq:lambda_z_relation} in the $\lambda$-update formula \eqref{eq:lambda-update} yields
\[
\lambda^{k-1} + \rho [A x^{k} + z^{k} - b] + \theta z^k + \tau_z \Delta z^k = 0,
\]
which is equivalent to $\nabla F(z^k) = 0$, where $F : z \mapsto \mathcal{L}(x^k, z, \lambda^{k-1}) + \frac{\tau_z}{2} \norm{z - z^{k-1}}^2$ is a convex quadratic function;
that is, $z^k$ minimizes $F$ and hence, $F(z^k) \leq F(z^{k-1})$ and this is seen to be equivalent to the $\alpha$-inequality.
The $\beta_t$-inequality is obtained by exploiting the fact $x_t^k$ is locally optimal with a better objective value than $x_t^{k-1}$ in problem~\eqref{eq:x-update}.
Finally, the $\gamma_t$-inequality follows from $a_1 a_2 \leq \frac12 (a_1^2 + a_2^2)$ for arbitrary real $a_1, a_2$.
The inequality \eqref{eq:delta_phi} is then obtained by combining \eqref{eq:lemma_lyapunov_descent_delta_phi}, \eqref{eq:lemma_lyapunov_descent_ab}, \eqref{eq:lemma_lyapunov_descent_a} and \eqref{eq:lemma_lyapunov_descent_b}.
\end{proof}

\begin{lemma}\label{lemma:lyapunov_lower_bound}
For all $k \geq 1$, we have ($\hat{\Phi}$ defined in \eqref{eq:lyapunov_lower_bound}):
    \[
    \Phi^{k-1} \geq \Phi^{k} \geq \hat{\Phi}.
    \]
\end{lemma}
\begin{proof}
    Since $\eta_x, \eta_z > 0$, Lemma~\ref{lemma:lyapunov_descent} implies that $\Phi^{k} - \Phi^{k-1} \leq 0$ for all $k \geq 1$.
    Now suppose, for the sake of contradiction, that there exists $l \geq 1$ such that $\Phi^l < \hat{\Phi}$.
    Therefore, we must have
    \begin{equation}\label{eq:lemma_lyapunov_lower_bound_temp}
    \Phi^k \leq \Phi^l < \hat{\Phi}, \; \forall k \geq l.
    \end{equation}
    Consider the augmented Lagrangian \eqref{eq:auglag_nlp} and Lyapunov functions \eqref{eq:lyapunov_function}.
    Here, all terms except $\sum_{t=1}^T f_t(x_t)$ and $\lambda^\top (Ax + z - b)$ are non-negative.
    Therefore, using \eqref{eq:lyapunov_lower_bound} and the $\lambda$-update formula \eqref{eq:lambda-update}, we have for all $j \geq 1$:
    \begin{align*}
    \Phi^j
    &\geq \hat{\Phi} + (\lambda^j)^\top \left(Ax^j + z^j - b \right) \\
    &\geq \hat{\Phi} + \frac{1}{\rho}(\lambda^j)^\top (\lambda^j - \lambda^{j-1}) \\ 
    &\geq \hat{\Phi} + \frac{1}{2\rho}\left(\norm{\lambda^j}^2 - \norm{\lambda^{j-1}}^2\right),
    \end{align*}
    where the inequality is obtained by noting that
    \[
    2 a_1^\top a_2 = \norm{a_1}^2 + \norm{a_2}^2 - \norm{a_1 - a_2}^2 \geq \norm{a_1}^2 - \norm{a_1 - a_2}^2,
    \]
    for arbitrary vectors $a_1, a_2$ and by setting $a_1 = \lambda^j$ and $a_2 = \lambda^j - \lambda^{j-1}$.
    Re-arranging the above inequality and summing over $j \in [k]$, we obtain:
    \[
    \sum_{j=1}^k \big(\Phi^j - \hat{\Phi}\big) \geq -\frac{1}{2\rho} \norm{\lambda^0}^2, \; \forall k \geq 1.
    \]
    However, \eqref{eq:lemma_lyapunov_lower_bound_temp} implies:
    \[
    \sum_{j=1}^k \big(\Phi^j - \hat{\Phi}\big) \leq \sum_{j=1}^l \big(\Phi^j - \hat{\Phi}\big) + (k - l) \big(\Phi^l - \hat{\Phi}\big), \; \forall k \geq l.
    \]
    Since $\Phi^l - \hat{\Phi} < 0$ (by hypothesis), the right-hand side of the above inequality can be made arbitrarily smaller than 0 for sufficiently large $k$, contradicting the previous inequality that $\sum_{j=1}^k \big(\Phi^j - \hat{\Phi}\big) \geq -(2\rho)^{-1} \norm{\lambda^0}^2$ for all $k \geq 1$.
\end{proof}

\begin{lemma}\label{lemma:delta_x_delta_z_bounds}
    For all $K \geq 1$, there exists $j \in [K]$ such that
    \begin{gather}
        \norm{\Delta z^j}^2 + \norm{\Delta z^{j-1}}^2 \leq \left(\Phi^1 - \Phi^K\right)\left(K\eta_z\right)^{-1}, \label{eq:delta_z_bound} \\
        \sum_{t=1}^T \norm{\Delta x_t^j}_{A_t^\top A_t}^2 \leq \left(\Phi^1 - \Phi^K\right)\left(K\eta_x\right)^{-1}. \label{eq:delta_x_bound} 
    \end{gather}
\end{lemma}
\begin{proof}
    Let $j$ be the index that maximizes the right-hand side of \eqref{eq:delta_phi} in Lemma \ref{lemma:lyapunov_descent} over $k \in [K]$. The first inequality follows by summing \eqref{eq:delta_phi} over $k \in [K]$ and noting that $\eta_x, \eta_z > 0$. The second one follows similarly if we ignore the contribution of $ \|\Delta x_t^{j-1} \|_{A_t^\top A_t} \geq 0$.
\end{proof}

\begin{proposition}\label{prop:primal_residual_bound}
    For all $K \geq 1$, there exists $j \in [K]$ such that the primal residual $\pi^j \coloneqq \pi(x^j)$, defined in \eqref{eq:residual_primal}, satisfies
    \begin{equation*}
        \pi^j \leq \sqrt{\frac{2\big(\Phi^1 - \hat{\Phi}\big)}{\theta} \left(1 + \frac{2(\theta + \tau_z)^2}{K\eta_z\rho} \right)}
    \end{equation*}
\end{proposition}
\begin{proof}
    First, we note that
    \begin{equation}\label{eq:prop_primal_residual_bound_temp}
    \Phi^1 \geq \Phi^k \geq \mathcal{L}(x^k, z^k, \lambda^k) \geq \hat{\Phi} + (c),
    \end{equation}
    where $(c)$ is equal to
    \begin{align*}
        \frac{\theta}{2} \norm{z^k}^2 + (\lambda^k)^\top \left[Ax^k + z^k - b\right] + \frac{\rho}{2} \norm{Ax^k + z^k - b}^2.
    \end{align*}
    Note that the first inequality in \eqref{eq:prop_primal_residual_bound_temp} follows from Lemma \ref{lemma:lyapunov_lower_bound};
    the second follows from the definition \eqref{eq:lyapunov_sequence_definition} of $\Phi^k$;
    and the third follows from the definitions \eqref{eq:auglag_nlp} and \eqref{eq:lyapunov_lower_bound} of the augmented Lagrangian function and $\hat{\Phi}$, respectively.
    Substituting $\lambda^k$ from \eqref{eq:lambda_z_relation} in the expression for $(c)$ above, we obtain
    \begin{align}
        (c) &= \frac{\theta}{2} \norm{z^k}^2 - (\theta z^k)^\top \left[Ax^k +z^k - b\right] + \frac{\rho}{2} \norm{Ax^k + z^k - b}^2 \notag\\
        &\phantom{=}\; -(\tau_z \Delta z^k)^\top \left[Ax^k + z^k - b\right]  \notag\\
        &= \frac{\theta}{2} \norm{Ax^k - b}^2 + \frac{\rho - \theta}{2} \norm{Ax^k + z^k - b}^2 \notag\\
        &\phantom{=}\; -(\tau_z \Delta z^k)^\top \left[Ax^k + z^k - b\right] \notag\\
        &\geq \frac{\theta}{2} \big(\pi^k\big)^2 -(\tau_z \Delta z^k)^\top \left[Ax^k + z^k - b\right] \notag\\
        &= \frac{\theta}{2} \big(\pi^k\big)^2 - (\tau_z\Delta z^k)^\top \big(\frac{1}{\rho} \Delta \lambda^k\big), \label{eq:prop_primal_residual_bound_c}
    \end{align}
    where the first inequality follows from $\rho > \theta$ which can be inferred by noting that $\eta_z > 0$ implies (after completing the square in $\tau_z$) that
    $
    (16\tau_z + 16\theta -\rho)^2 + (16 \theta)^2 - (16\theta -\rho)^2 < 0,
    $
    implying $(16 \theta)^2 - (16\theta -\rho)^2 < 0$,
    and hence, $\rho(\rho - 32\theta) > 0$, that is, $\rho > 32\theta > \theta$.
    Equation~\eqref{eq:prop_primal_residual_bound_c} then follows directly from the $\lambda$-update formula \eqref{eq:lambda-update}. %
    The second term in the right-hand side of \eqref{eq:prop_primal_residual_bound_c} can be bounded using: \emph{(i)} the Cauchy-Schwarz inequality, \emph{(ii)} $a_1 a_2 \leq \frac12 (a_1^2 + a_2^2)$ for real $a_1, a_2$, and \emph{(iii)} relation \eqref{eq:lemma_lyapunov_descent_a} from Lemma \ref{lemma:lyapunov_descent}, as follows:
    \begin{align}
        &(\tau_z\Delta z^k)^\top \big(\frac{1}{\rho} \Delta \lambda^k\big) \notag\\
        & \leq \frac{1}{2\rho}\Big( \tau_z^2\norm{\Delta z^k}^2 +  \norm{\Delta \lambda^k}^2 \Big) \notag\\
        &\leq \frac{(\theta + \tau_z)^2}{\rho}\Big( \norm{\Delta z^k}^2 +  \left(\norm{\Delta z^k}^2 + \norm{\Delta z^{k-1}}^2\right) \Big). \label{eq:prop_primal_residual_bound_temp_2}
    \end{align}
    The claim now follows from inequalities \eqref{eq:prop_primal_residual_bound_temp}, \eqref{eq:prop_primal_residual_bound_c}, \eqref{eq:prop_primal_residual_bound_temp_2}, and from the $\Delta z^k$-bound \eqref{eq:delta_z_bound} established in Lemma \ref{lemma:delta_x_delta_z_bounds}.
\end{proof}

\begin{proposition}\label{prop:dual_residual_bound}
    For all $K \geq 1$, there exists $j \in [K]$ such that the dual residual $\delta_t^j \coloneqq \delta_t(x_t^j, \lambda^j)$, defined in \eqref{eq:residual_dual}, satisfies
    \begin{equation}
        \delta_t^j \leq (\rho + \tau_x) \norm{A_t} \sqrt{\frac{(T+1)\big(\Phi^1 - \Phi^K\big)}{K \min\{\eta_x, \eta_z\}}} %
    \end{equation}
    for all $t \in [T]$.
\end{proposition}
\begin{proof}
    Fix $t \in [T]$. Since $x_t^{k}$ is locally optimal in problem~\eqref{eq:x-update}, the first-order optimality conditions imply, for $k \geq 1$:
    \begin{equation*}
    \left(
       \begin{aligned}
           \nabla f_t(x_t^k) + A_t^\top \lambda^{k-1} + \tau_x A_t^\top A_t \Delta x_t^k \\
           + \rho A_t^\top \left[ A_tx_t^k + A_{\neq t} x_{\neq t}^{k-1} + z^{k-1} - b \right]
       \end{aligned}
    \right)
     \in -N_{X_t}(x_t^k).
    \end{equation*}
    Substituting $\lambda^{k-1} = \lambda^k - \rho [Ax^k + z^k - b]$ from \eqref{eq:lambda-update}:
    \begin{equation*}
        \left(
        \begin{aligned}
            \nabla f_t(x_t^k) + A_t^\top \lambda^{k} + \tau_x A_t^\top A_t \Delta x_t^k \\
            -\rho A_t^\top A_{\neq t} \Delta x_{\neq t}^k - \rho A_t^\top \Delta z^k
        \end{aligned}
        \right)
        \in -N_{X_t}(x_t^k).
    \end{equation*}
    Definition \eqref{eq:residual_dual} of
    $\delta_t^k = \dist{\nabla f_t(x_t^k) + A_t^\top \lambda^k}{-N_{X_t}(x_t^k)}$
    implies that the latter quantity is bounded from above by:
    \begin{align*}
        &(\delta_t^k)^2
        \leq \norm{-\rho A_t^\top A_{\neq t} \Delta x_{\neq t}^k - \rho A_t^\top \Delta z^k + \tau_x A_t^\top A_t \Delta x_t^k}^2 \\
        &\leq \norm{A_t}^2 \left(  \rho \norm{A_{\neq t} \Delta x_{\neq t}^k} +  \rho \norm{\Delta z^k} +  \tau_x \norm{A_t \Delta x_t^k} \right)^2 \\
        &\leq (T+1)(\rho + \tau_x)^2\norm{A_t}^2 \Big(\sum_{s=1}^T \norm{\Delta x_s^k}_{A_s^\top A_s}^2 + \norm{\Delta z^k}^2 \Big).
    \end{align*}
    where the inequality on the second line follows from the Cauchy-Schwarz and triangle inequalities, and the last inequality follows from $\max\{\rho, \tau_x\} \leq (\rho + \tau_x)$ and $(a_1 + \ldots + a_N)^2 \leq N (a_1^2 + \ldots a_N^2)$ for arbitrary real $a_1, \ldots, a_N$.
    The claim now follows from the bounds \eqref{eq:delta_z_bound}, \eqref{eq:delta_x_bound} in Lemma \ref{lemma:delta_x_delta_z_bounds}.
\end{proof}

\section{Local convergence proof}\label{sec:local_convergence_proof}
Throughout this section, we suppose that the conditions outlined in Theorem~\ref{theorem:local_convergence} hold.
Also, for $k \geq 1$, we define $\mu^k \coloneqq (\mu_1^k, \ldots, \mu_T^k)$, where $\mu_t^k$ is the optimal Lagrange multiplier vector of the $x_t$-subproblem~\eqref{eq:x-update} at iteration $k$.

The key steps of the proof are as follows.
We first introduce the following perturbed variant of problem~\eqref{eq:nlp}
that is parameterized by $p \in \mathbb{R}^m$ and $d = (d_1, \ldots, d_T, d_z) \in \mathbb{R}^{n+m}$.
\begin{equation}\label{eq:nlp_perturbed}\tag{$S(p, d)$}
    \begin{aligned}
        \displaystyle\mathop{\text{minimize}}_{x_1, \ldots, x_T, z} \;\;& \displaystyle \sum_{t = 1}^T \left[ f_t(x_t) - d_t^\top x_t \right] + \frac{\theta}{2} \norm{z}^2 - d_z^\top z \\
        \text{subject to} \;\;& \displaystyle c_t(x_t) = 0, \;\; t \in [T], \\
        & \displaystyle Ax  + z = b + p.
    \end{aligned}
\end{equation}
Observe that $S(0,0)$ coincides precisely with problem~\eqref{eq:nlp}.
By defining $p^k$ and $d^k$ to be the primal and dual residuals of problem~\eqref{eq:nlp} at $(x^k, z^k)$ as per eq.~\eqref{eq:perturbation}, Lemma \ref{lemma:nlp_sensitivity} and \ref{lemma:iterates_and_nlp_perturbed} show that $(x^k, z^k, \mu^k, \lambda^k)$ must  coincide precisely with the optimal solution of $S(p^k, d^k)$, whenever the former is close to $(x^*, z^*, \mu^*, \lambda^*)$ and $p^k$, $d^k$ are sufficiently close to $0$.
This allows us to bound the difference in the augmented Lagrangian function values of problem~\eqref{eq:nlp} evaluated at its optimal solution and at an arbitrary iterate (Lemma~\ref{lemma:lower_bound} and \ref{lemma:upper_bound}).
In Lemma~\ref{lemma:subproblem_sensitivity}, we establish that it is sufficient to have $(Dx^k, z^k, \lambda^k)$ to be close to $(Dx^*, z^*, \lambda^*)$.
Finally, Proposition~\ref{prop:local_convergence} proves that whenever the latter is true, then the sequence $\{(x^k, z^k)\}$ must converge to $(x^*, z^*)$.
This requires some intermediate results that are proved in Lemma \ref{lemma:convergence_of_residuals} and \ref{lemma:intermediate}.
The proof of Theorem~\ref{theorem:local_convergence} follows directly from that of Proposition~\ref{prop:local_convergence}.
Indeed, its first part is already true due to Theorem~\ref{theorem:global_convergence}, whereas its second part follows from Assumption~\ref{assume:sosc} if we set $w_z = 0$ therein.

\begin{lemma}\label{lemma:nlp_sensitivity}
    There exist $C^1$ functions $\hat{x}$, $\hat{z}$, $\hat{\mu}$, $\hat{\lambda}$ with domain $\mathbb{R}^m \times \mathbb{R}^{n+m}$
    and constants $\epsilon_1, \epsilon_2 > 0$ such that
    \begin{enumerate}
        \item $(\hat{x}(0, 0), \hat{z}(0, 0), \hat{\mu}(0, 0), \hat{\lambda}(0, 0)) = (x^*, z^*, \mu^*, \lambda^*)$.
        \item For all $(p, d) \in B_{\epsilon_1}(0)$,
        $(\hat{x}(p, d), \hat{z}(p, d))$ is a local minimizer of~\ref{eq:nlp_perturbed}
        where the LICQ and SOSC assumptions are satisfied
        with optimal Lagrange multipliers $(\hat{\mu}(p, d), \hat{\lambda}(p, d))$.
        \item The functions
        $\hat{x}$, $\hat{z}$, $\hat{\mu}$, $\hat{\lambda}$ map to locally unique points in the sense that,
        for all $(p, d) \in B_{\epsilon_1}(0)$,
        if
        $(\bar{x}, \bar{z})$ is a local minimizer (or a first-order stationary point) of \ref{eq:nlp_perturbed} with optimal Lagrange multipliers $(\bar{\mu}, \bar{\lambda})$,
        and if $(\bar{x}, \bar{z}, \bar{\lambda}, \bar{\mu}) \in B_{\epsilon_2}((x^*, z^*, \mu^*, \lambda^*))$,
        then
        $(\bar{x}, \bar{z}, \bar{\lambda}, \bar{\mu}) = (\hat{x}(p, d), \hat{z}(p, d), \hat{\mu}(p, d), \hat{\lambda}(p, d))$.
    \end{enumerate}
\end{lemma}
\begin{proof}
    The LICQ \ref{assume:licq} and SOSC assumptions \ref{assume:sosc} are satisfied at $(x^*, z^*)$,
    The claims then follow from a direct application of \cite[Theorem~2.1]{fiacco1976sensitivity} to \ref{eq:nlp_perturbed}.
\end{proof}

\begin{lemma}\label{lemma:iterates_and_nlp_perturbed}
    For all $k \geq 1$ such that $(p^k, d^k) \in B_{\epsilon_1}(0)$ and $(x^k, z^k, \mu^k, \lambda^k) \in B_{\epsilon_2}((x^*, z^*, \mu^*, \lambda^*))$, where $\epsilon_1, \epsilon_2$ are defined in Lemma~\ref{lemma:nlp_sensitivity} and $p^k$ and $d^k$ from \eqref{eq:perturbation}, we have:
    \begin{equation*}
        (\hat{x}(p^k, d^k), \hat{z}(p^k, d^k), \hat{\mu}(p^k, d^k), \hat{\lambda}(p^k, d^k)) = (x^k, z^k, \mu^k, \lambda^k).
    \end{equation*}
\end{lemma}
\begin{proof}
    We first show that $(x^k, z^k, \mu^k, z^k)$ satisfy the KKT conditions of $S(p^k, d^k)$.
    To see this, first note that $x_t^k$ is locally optimal in problem~\eqref{eq:x-update}.
    Under assumption~\ref{assume:licq:subproblem}, this implies that it satisfies the first-order optimality condition:
    \begin{align*}
        \left(\begin{aligned}
            \nabla f_t(x_t^k) + \nabla c_t(x_t^k) \mu_t^k + A_t^\top \lambda^{k-1} + \tau_x A_t^\top A_t \Delta x_t^k \\
            + \rho A_t^\top \left[ A_tx_t^k + A_{\neq t} x_{\neq t}^{k-1} + z^{k-1} - b \right]
        \end{aligned}\right) = 0.
    \end{align*}
    Using the $\lambda$-update~\eqref{eq:lambda-update} and eq.~\eqref{eq:perturbation:dual_xt}, this is equivalent to:
    \begin{align*}
        \nabla f_t(x_t^k) + \nabla c_t(x_t^k) \mu_t^k + A_t^\top \lambda^{k} - d_t^k = 0.
    \end{align*}
    Second, note that Lemma~\ref{lemma:lambda_z_relation} along with equation~\eqref{eq:perturbation:dual_z} implies:
    \begin{align*}
        \theta z^k + \lambda^k = d_z^k.
    \end{align*}
    Finally, equation~\eqref{eq:perturbation:primal} can be equivalently written as follows:
    \begin{align*}
        Ax^k + z^k = b + p^k.
    \end{align*}
    Along with $c_t(x_t^k) = 0$, the last three equations are precisely the KKT conditions of $S(p^k, d^k)$.
    The result now follows from the third part of Lemma~\ref{lemma:nlp_sensitivity}.
\end{proof}

\begin{lemma}\label{lemma:lower_bound}
    There exist constants $\rho_1, \epsilon_3 > 0$ such that for all $\rho > \rho_1$ and $k \geq 1$ where $(x^k, z^k) \in B_{\epsilon_3}((x^*, z^*))$, we have:
    \begin{equation}\label{eq:lower_bound}
        \begin{aligned}
        &\sum_{t \in [T]} f_t(x_t^k) + \frac{\theta}{2}\norm{z^k}^2
        +(\lambda^*)^\top p^k + \frac{\rho}{4}\norm{p^k}^2 \\
        &\;\;\geq
        \sum_{t \in [T]} f_t(x_t^*) + \frac{\theta}{2}\norm{z^*}^2.
        \end{aligned}
    \end{equation}
\end{lemma}
\begin{proof}
    Define the (fully) augmented Lagrangian function of $S(0, 0)$ with respect to the optimal multipliers $(\mu^*, \lambda^*)$, as follows $\mathcal{S} : \mathbb{R}^n \times \mathbb{R}^{m} \mapsto \mathbb{R}$, where $\mathcal{S}(x, z)$ is given by:
    \begin{equation*}%
        \Lambda_\theta(x, z, \mu^*, \lambda^*) + \frac{\rho}{4} \sum_{t = 1}^T \norm{c_t(x_t)}^2 + \frac{\rho}{4} \norm{Ax + z - b}^2
    \end{equation*}
    The definition of $(x^*, z^*)$ as a local minimizer of problem $S(0,0)$ satisfying the LICQ assumption, means that $\nabla \mathcal{S}(x^*, z^*) = 0$ and $c_t(x^*_t) = 0$.
    The latter also implies that the Hessian $\nabla^2 \mathcal{S}(x^*, z^*)$ is given by:
    \[
    \begin{bmatrix}
        \nabla^2_{xx} \Lambda(x^*, z^*, \mu^*, \lambda^*) & 0 \\
        0 & \theta \mathrm{I}
    \end{bmatrix}
    + \frac{\rho}{2}
    \begin{bmatrix}
        J(x^*)^\top J(x^*) & A^\top \\
        A & \mathrm{I}
    \end{bmatrix}.
    \]
    Assumption~\ref{assume:sosc} and \cite[Lemma~3.2.1]{bertsekas1999nonlinear} ensure the existence of $\rho_1 > 0$ such that
    $\nabla^2 \mathcal{S}(x^*, z^*)$ is positive definite and that its minimum eigenvalue is larger than some fixed $\epsilon_3' > 0$,
    for all $\rho > \rho_1$.
    Continuity of $\nabla^2 \mathcal{S}$ along with the second-order sufficient conditions for unconstrained minimization then imply the existence of $\epsilon_3 > 0$ such that
    $(x^*, z^*)$ is a local minimizer of $\mathcal{S}$ in some neighborhood of radius $\epsilon_3$ around
    $(x^*, z^*)$. %
    Note that the radius $\epsilon_3$ is independent of $\rho$ as long as the minimum eigenvalue of $\nabla^2 \mathcal{S}(x^*, z^*)$ remains larger than $\epsilon_3' > 0$.
    The statement of the lemma then follows by noting that $\mathcal{S}(x^k, z^k) \geq \mathcal{S}(x^*, z^*)$ and after substituting the expressions for $\mathcal{S}(x, z)$ and for $p^k$ from~\eqref{eq:perturbation:primal}.
\end{proof}

\begin{lemma}\label{lemma:upper_bound}
    There exist constants $\rho_2, \epsilon_4 > 0$ such that for all $\rho > \rho_2$ and $k \geq 1$ where $(p^k, d^k) \in B_{\epsilon_4}(0)$ and $(x^k, z^k, \mu^k, \lambda^k) \in B_{\epsilon_2}((x^*, z^*, \mu^*, \lambda^*))$, with $\epsilon_2$ defined in Lemma~\ref{lemma:nlp_sensitivity}, we have:
    \begin{equation}\label{eq:upper_bound}
        \begin{aligned}
            &\sum_{t \in [T]} \big(f_t(x_t^k) - d_t^k x_t^k \big) + \frac{\theta}{2}\norm{z^k}^2 - (d^k_z)^\top z^k
            +(\lambda^k)^\top p^k - \frac{\rho}{4}\norm{p^k}^2 \\
            &\;\;\leq
            \sum_{t \in [T]} \big(f_t(x_t^*) - d_t^k x_t^* \big) + \frac{\theta}{2}\norm{z^*}^2 - (d^k_z)^\top z^*.
        \end{aligned}
    \end{equation}
\end{lemma}
\begin{proof}
    Define the primal functional corresponding to \ref{eq:nlp_perturbed}, as follows $\mathcal{Q} : \mathbb{R}^m \times \mathbb{R}^{n+m} \mapsto \mathbb{R}$, where $\mathcal{Q}(p, d)$ is given by:
    \begin{equation*}%
        \begin{aligned}
            \sum_{t = 1}^T \left[ f_t(\hat{x}_t(p, d)) - d_t^\top \hat{x}_t(p, d) \right] + \frac{\theta}{2} \norm{\hat{z}(p, d)}^2 - d_z^\top \hat{z}(p, d) \; .
        \end{aligned}
    \end{equation*}
    Now \cite[Proposition~1.28]{bertsekas2014constrained} implies that $\nabla_p Q(p, d) = -\hat{\lambda}(p, d)$, which is continuously differentiable from Lemma~\ref{lemma:nlp_sensitivity}. Therefore, $\nabla^2_{pp} Q$ is continuous and \cite[Lemma~3.2.1]{bertsekas1999nonlinear} ensures the existence of $\rho_2 > 0$ such that $\nabla^2_{pp} Q(0, 0) + \frac{\rho}{2} \mathrm{I}$ is positive definite {and that its minimum eigenvalue is bounded strictly away from $0$} for all $\rho > \rho_2$.
    By a similar argument as in the proof of Lemma~\ref{lemma:lower_bound},
    continuity of $\nabla^2_{pp} Q$ implies that there exists $\epsilon_4' > 0$ (independent of $\rho$) such that $\nabla^2_{pp} Q(p, d) + \frac{\rho}{2} \mathrm{I}$ is positive definite, and hence the function $F_d: p \mapsto Q(p, d) + \frac{\rho}{4} \norm{p}^2$ is convex, whenever $(p, d) \in B_{\epsilon_4'}(p, d)$. %
    Now fix $\epsilon_4 = \min\{\epsilon_4', \epsilon_1\} > 0$ and $d = d^k$.
    Convexity of $F_{d_k}$ implies:
    \[
    F_{d_k}(p^k) + (0 - p^k)^\top\nabla F_{d_k}(p^k) \leq F_{d_k}(0).
    \]
    The statement of the lemma now follows by
    \emph{(i)} substituting
    $F_{d_k}(p^k) = Q(p^k, d^k) + \frac{\rho}{4} \norm{p^k}^2$, replacing $Q(p^k, d^k)$ using its definition,
    and noting $(\hat{x}(p^k, d^k), \hat{z}(p^k, d^k)) = (x^k, z^k)$ from Lemma~\ref{lemma:iterates_and_nlp_perturbed};
    \emph{(ii)} substituting
    $\nabla F_{d_k}(p^k) = -\hat{\lambda}(p^k, d^k) + \frac{\rho}{2} p^k = -\lambda^k + \frac{\rho}{2} p^k$ (from Lemma~\ref{lemma:iterates_and_nlp_perturbed});
    and,
    \emph{(iii)} $F_{d_k}(0) = Q(0, d^k)$ is less than or equal to the right-hand side of \eqref{eq:upper_bound},
    since $(x^*, z^*)$ is feasible (but possibly suboptimal) in $S(0, d^k)$.
\end{proof}

\begin{lemma}\label{lemma:subproblem_sensitivity}
    There exist $\rho_3, \epsilon_5 > 0$ such that for all $\rho > \rho_3$ and $k \geq 1$ where $(Dx^{k-1}, z^{k-1}, \lambda^{k-1}) \in B_{\epsilon_5}((Dx^*, z^*, \lambda^*))$, we have
    $ %
    (x^k, z^k, \mu^k, \lambda^k) \in B_{\min\{\epsilon_2, \epsilon_3\}}((x^*, z^*, \mu^*, \lambda^*)),
    $ %
    where $\epsilon_2, \epsilon_3$ are defined in Lemma~\ref{lemma:nlp_sensitivity} and~\ref{lemma:lower_bound}.
\end{lemma}
\begin{proof}
    Observe that for fixed choices of $\rho, \theta, \tau_x, \tau_z$ in Algorithm~\ref{algo:main}, the $x$-update step \eqref{eq:x-update} at iteration $k$ depends only on the values of $Dx^{k-1}$, $z^{k-1}$ and $\lambda^{k-1}$ at the previous iteration;
    therefore, we can define
    $G : (Dx^{k-1}, z^{k-1}, \lambda^{k-1}) \mapsto (x^k, \mu^k)$
    to be the corresponding mapping.
    Observe now that the claim follows trivially if we can show that
    $G$ is continuous and satisfies
    $G(Dx^*, z^*, \lambda^*) = (x^*, \mu^*)$.
    Indeed, if this is true, then observe that the $z$-update \eqref{eq:z-update} and $\lambda$-update \eqref{eq:lambda-update} formulas also have the same properties: they define continuous maps and satisfy $z^k = z^*$ and $\lambda^k = \lambda^*$ whenever $x^k = x^*$, $z^{k-1} = z^*$ and $\lambda^{k-1} = \lambda^*$, since $Ax^* + z^* = b$ and $\lambda^* = -\theta z^*$ (from the KKT conditions of problem~\eqref{eq:nlp}).
    
    Now fix $t \in [T]$ and $(Dx^{k-1}, z^{k-1}, \lambda^{k-1}) = (Dx^*, z^*, \lambda^*)$.
    Observe that problem~\eqref{eq:x-update} satisfies the LICQ assumption at $x_t^*$, because of~\ref{assume:licq}.
    Moreover, since $Ax^* + z^* = b$, observe that $(x_t^*, \mu_t^*)$ is also a first-order stationary point of problem~\eqref{eq:x-update}:
    \begin{align*}
        \left(\begin{aligned}
            \nabla f_t(x_t^*) + \nabla c_t(x_t^*) \mu_t^* + A_t^\top \lambda^{*} + \tau_x A_t^\top A_t (x_t^* - x_t^*) \\
            + \rho A_t^\top \left[ A_tx_t^* + A_{\neq t} x_{\neq t}^{*} + z^{*} - b \right]
        \end{aligned}\right) = 0.
    \end{align*}
    To show that it satisfies SOSC, note that Assumption~\ref{assume:sosc} and \cite[Lemma~3.2.1]{bertsekas1999nonlinear} ensure the existence of $\rho_3 > 0$ such that
    \[
    \begin{bmatrix}
        \nabla^2_{xx} \Lambda(x^*, z^*, \mu^*, \lambda^*) & 0 \\
        0 & \theta \mathrm{I}
    \end{bmatrix}
    + \rho
    \begin{bmatrix}
        A^\top A & A^\top \\
        A & \mathrm{I}
    \end{bmatrix}.
    \]
    is positive definite on the domain $\{(w_x, w_z) \in \mathbb{R}^n \times \mathbb{R}^m \setminus \{0\}: \nabla c(x^*)^\top w_x = 0\}$ for all $\rho > \rho_3$.
    This means that for $w_{x_t} \neq 0$, $\nabla c_t(x_t^*)^\top w_{x_t} = 0$, $w_z = 0$, and $w_{x_s} = 0$ for $s \neq t$, we have:
    \[
    w_{x_t}^\top
    \left[
    \nabla^2_{x_t x_t} \Lambda(x^*, z^*, \mu^*, \lambda^*)
    + (\rho + \tau_x) A_t^\top A_t
    \right] w_{x_t} > 0,
    \]
    which is precisely the SOSC condition for problem~\eqref{eq:x-update} at $(x_t^*, \mu_t^*)$, and therefore, $(x^*, \mu^*)$ is a strict local solution.
    
    Continuity of the mapping $G$ now follows directly from Assumption~\ref{assume:subproblem:regular} and classical NLP sensitivity~\cite[Theorem~2.1]{fiacco1976sensitivity} applied to problem~\eqref{eq:x-update} with $Dx_t^{k-1}$, $z^{k-1}$ and $\lambda^{k-1}$ viewed as perturbation parameters.
\end{proof}

\begin{lemma}\label{lemma:convergence_of_residuals}
    The sequences $\{p^k\}$ and $\{d^k\}$ converge to $0$.
\end{lemma}
\begin{proof}
    After summing inequality~\eqref{eq:lemma_lyapunov_descent_delta_phi} over $k \in [K]$ and noting $\Phi^K \geq \hat{\Phi}$ from Lemma~\ref{lemma:lyapunov_lower_bound}, we obtain that the sequences $\{\Delta z^k\}$ and $\{A_t \Delta x_t^k\}$ for $t \in [T]$ must all converge to $0$.
    Therefore, the definitions \eqref{eq:perturbation:dual_xt} and \eqref{eq:perturbation:dual_z} imply that $\{d^k\}$ also converges to $0$.
    Equations \eqref{eq:lambda-update} and \eqref{eq:delta_lambda} imply
    $p^k = \frac{1}{\rho}\Delta \lambda^k = -\frac{\theta + \tau_z}{\rho} \Delta z^k + \frac{\tau_z}{\rho} \Delta z^{k-1}$,
    which means that the sequence $\{p^k\}$ must also converge to $0$.
\end{proof}

\begin{lemma}\label{lemma:intermediate}
    \begin{enumerate}
        \item Any two vectors $a_1, a_2$ of equal dimension satisfy
        $
        2 a_1^\top a_2 = \norm{a_1 + a_2}^2 - \norm{a_1}^2 - \norm{a_2}^2
        $.
        \item For any $w = (w_1, \ldots, w_T) \in \mathbb{R}^{n_1} \times \ldots \times  \mathbb{R}^{n_T}$, we have:
        \begin{equation*}%
            \begin{aligned}
                &(\rho+\tau_x) \sum_{t = 1}^T \norm{w_t}^2_{A_t^\top A_t} - \rho \norm{w}^2_{A^\top A} = \norm{Dw}^2_R,  \\
                &R \coloneqq (\rho + \tau_x) \mathrm{I} - \rho E E^\top  \succ 0, \\
                & E^\top \coloneqq \begin{bmatrix}
                    \mathrm{I}_{m \times m} & \ldots & \mathrm{I}_{m\times m}
                \end{bmatrix} \in \mathbb{R}^{m \times Tm}.
            \end{aligned}
        \end{equation*}
    \end{enumerate}
\end{lemma}
\begin{proof}
    The first part follows by expanding its right-hand side.
    To prove the second part, note that $D = \mathop{\mathrm{diag}}(A_1, \ldots, A_T)$ can be used to verify that
    $
    \sum_{t = 1}^T \norm{w_t}^2_{A_t^\top A_t} = w^\top D^\top D w.
    $
    Similarly, the definition of $E$ can be used to verify that $A = E^\top D$ and hence, 
    $
    \norm{w}^2_{A^\top A} = w^\top D^\top E E^\top D w.
    $
    This proves the claimed equation.
    To show that $R \succ 0$, observe that $E\; E^\top$ can be equivalently expressed as the Kronecker product
    $\mathrm{e} \; \mathrm{e}^\top \otimes  \mathrm{I}_{m \times m}$,
    where $\mathrm{e}$ is the vector of ones in $\mathbb{R}^T$.
    Therefore, its eigenvalues are given by pairwise products of the eigenvalues of $\mathrm{e} \; \mathrm{e}^\top$ (which are $0$ and $T$)
    and $\mathrm{I}_{m \times m}$ (which is $1$).
    Therefore, the eigenvalues of $E \; E^\top$ are $0$ and $T$, and hence, those of $(\rho + \tau_x) \mathrm{I} - \rho E\;  E^\top$ are $\rho + \tau_x$ and $\rho + \tau_x - \rho T$, both of which are positive since $\eta_x > 0$, see \eqref{eq:eta_definition}.
\end{proof}

\begin{proposition}\label{prop:local_convergence}
    There exist constants $\bar{\rho}, \epsilon' > 0$ such that if $\rho > \bar{\rho}$ and
    $
    \norm{(Dx^0, z^0, \lambda^0) - (Dx^*, z^*, \lambda^*)} < \epsilon'
    $,
    then the sequence $\{(x^k, z^k, \lambda^k, \Delta z^k)\}$ converges to $(x^*, z^*, \lambda^*, 0)$ at a sublinear rate.
\end{proposition}
\begin{proof}%
    Set $\bar{\rho} = \max\{\rho_1, \rho_2, \rho_3\}$, where the latter are defined in Lemma \ref{lemma:lower_bound}, \ref{lemma:upper_bound}, \ref{lemma:subproblem_sensitivity}.
    Also, assume for the moment that the other conditions of Lemma \ref{lemma:lower_bound} and \ref{lemma:upper_bound} are also satisfied for all $k \geq 1$; we shall shortly show how $\epsilon'$ can be chosen to ensure this.
    Adding the corresponding equations~\eqref{eq:lower_bound} and~\eqref{eq:upper_bound}, we obtain:
    \begin{equation*}
        \begin{aligned}
        &\underset{(d)}{\underbrace{(\lambda^* - \lambda^k)^\top p^k + \frac{\rho}{2}\norm{p^k}^2}}
        +\underset{(e)}{\underbrace{(d_z^k)^\top (z^k - z^*)}} \\
       &\quad +\underset{(f)}{\underbrace{\sum_{t = 1}^T (d_t^k)^\top (x_t^k - x_t^*)}}
       \end{aligned} \geq 0.
    \end{equation*}
   We now we have the following
    \begin{align*}
        (d) %
        &= \frac{1}{\rho} (\lambda^* - \lambda^k)^\top \Delta \lambda^k + \frac{1}{2\rho} \norm{\Delta \lambda^k}^2 \\
        &= \frac{1}{2\rho}\big[ \norm{\lambda^* - \lambda^{k-1}}^2 - \norm{\lambda^* - \lambda^k}^2  \big],
    \end{align*}
    where the first equality follows from definition \eqref{eq:perturbation:primal} and the $\lambda$-update formula \eqref{eq:lambda-update}, and the second equality follows from the first part of Lemma~\ref{lemma:intermediate}. Subsequently, 
    \begin{align*}
        (e) &= %
        \tau_z (\Delta z^k)^\top (z^* - z^k) \\
        &= \frac{\tau_z}{2} \big[ \norm{z^* - z^{k-1}}^2 - \norm{z^* - z^k}^2 - \norm{\Delta z^k}^2 \big],
    \end{align*}
    where again the first equality follows from definition \eqref{eq:perturbation:dual_z} and the second follows from the first part of Lemma~\ref{lemma:intermediate}. Finally, 
    \begin{align*}
        (f) %
        &= \sum_{t = 1}^T \big[ \rho A \Delta x^k + \rho \Delta z^k - (\rho + \tau_x) A_t \Delta x_t^k \big]^\top A_t (x_t^k - x_t^*) \\
        &= (f_1) + (f_2),
    \end{align*}
    where the first equality follows from definition \eqref{eq:perturbation:dual_xt} and after adding and subtracting $\rho A_t \Delta x_t^k$ from the latter. For the first component of $(f)$ we have
    \begin{align*}
        (f_1) &= -\rho (A \Delta x^k)^\top A (x^* - x^k) \\
        &\phantom{=} + (\rho + \tau_x) \sum_{t=1}^T (A_t \Delta x_t^k)^\top A_t(x_t^* - x_t^k) \\
        =& -\frac{\rho}{2} \left[ \norm{x_t^* - x_t^{k-1}}_{A^\top A}^2 - \norm{x_t^* - x_t^{k}}_{A^\top A}^2 - \norm{\Delta x_t^k}_{A^\top A}^2 \right] \\
        & + \frac{\rho+\tau_x}{2} \sum_{t=1}^T \left[
        \begin{aligned}
            & \norm{x_t^* - x_t^{k-1}}^2_{A_t^\top A_t} - \norm{x_t^* - x_t^{k}}^2_{A_t^\top A_t} \\
            & - \norm{\Delta x_t^k}^2_{A_t^\top A_t}
        \end{aligned} \right] \\
        =& \frac{1}{2} \big[ \norm{D(x^* - x^{k-1})}_R^2 - \norm{D(x^* - x^k)}_R^2 - \norm{D\Delta x^k}_R^2 \big] %
    \end{align*}
    where the second and third equalities follow from the first and second parts of Lemma~\ref{lemma:intermediate}, respectively. For the second component of $f$ we have
    \begin{align*}
        (f_2) &= \rho (\Delta z^k)^\top A (x^k - x^*) \\ %
        &= (\Delta z^k)^\top \Delta \lambda^k + \rho (\Delta z^k)^\top (z^* - z^k) \\
        &= -(\theta + \tau_z) \norm{\Delta z^k}^2 + \tau_z (\Delta z^k)^\top \Delta z^{k-1} \\
        &\phantom{=} +\frac{\rho}{2} \big[ \norm{z^* - z^{k-1}}^2 - \norm{z^* - z^k}^2 - \norm{\Delta z^k}^2 \big] \\
        &\leq \frac{\rho}{2} \big[ \norm{z^* - z^{k-1}}^2 - \norm{z^* - z^k}^2 \big] + \frac{\tau_z}{2} \norm{\Delta z^k}^2,
    \end{align*}
    where the second equality follows by substituting $Ax^* = b - z^*$ and $Ax^k = b - z^k + \frac{1}{\rho}\Delta \lambda^k$,
    the third equality follows from \eqref{eq:delta_lambda} and the first part of Lemma~\ref{lemma:intermediate},
    and the inequality follows by noting first that $2 (\Delta z^k)^\top \Delta z^{k-1} \leq \|{\Delta z^k} \|^2 + \| {\Delta z^{k-1}} \|^2$ and then that $\theta + (\tau_z/2) + \rho > 0$.
    
    Combining the equations for $(d), (e), (f), (f_1)$, the inequality for $(f_2)$, and $(d) + (e) + (f) \geq 0$, we obtain:
    \begin{equation}\label{eq:local:contraction}
        \begin{aligned}
    &\norm{(Dx^k, z^k, \lambda^k) - (Dx^*, z^*, \lambda^*)}^2_{*} + \tau_z \norm{\Delta z^k}^2 + \norm{D \Delta x^k}^2_R \\
    &\leq
    \norm{(Dx^{k-1}, z^{k-1}, \lambda^{k-1}) - (Dx^*, z^*, \lambda^*)}^2_{*} + \tau_z \norm{\Delta z^{k-1}}^2,
    \end{aligned}
    \end{equation}
    where we define the norm:
    \[
    \norm{(Dx, z, \lambda)}_{*} \coloneqq \sqrt{\norm{Dx}^2_{R} + (\rho + \tau_z)\norm{z}^2 + (1/\rho) \norm{\lambda}^2 }.
    \]
    Summing inequality \eqref{eq:local:contraction} over $k \in [j]$ and using
    $\Delta z^0 = (-\lambda^0 - \theta z^0)/\tau_z = 0$ (by hypothesis),
    we obtain for all $j \geq 1$:
    \begin{equation*}
        \begin{aligned}
            &\norm{(Dx^j, z^j, \lambda^j) - (Dx^*, z^*, \lambda^*)}^2_{*} + \tau_z \norm{\Delta z^j}^2 + \norm{D \Delta x^j}^2_R \\
            &\leq
            \norm{(Dx^0, z^0, \lambda^0) - (Dx^*, z^*, \lambda^*)}^2_{*}
        \end{aligned}
    \end{equation*}
    The equivalence of norms implies that whenever the Euclidean norm
    $\norm{(Dx^0, z^0, \lambda^0) - (Dx^*, z^*, \lambda^*)}$
    is less than $\epsilon'$,
    the right-hand side of the above inequality is also sufficiently small,
    and the terms on the left-hand side are even smaller.
    In particular, by choosing a small $\epsilon'$, we can ensure:
    \emph{(i)} $(Dx^j, z^j, \lambda^j)$ remains close to $(Dx^*, z^*, \lambda^*)$; and
    \emph{(ii)} $D\Delta x^j$ and $\Delta z^j$ are close to $0$, which implies
    $p^j, d^j$ are also close to $0$
    (see argument in proof of Lemma \ref{lemma:convergence_of_residuals}).
    Finally, note that \textit{(i)} and \textit{(ii)} satisfy the conditions of Lemma~\ref{lemma:subproblem_sensitivity} (which in turn allows us to satisfy those of Lemma~\ref{lemma:lower_bound}) and Lemma~\ref{lemma:upper_bound}, respectively.
    
    To be precise, the equivalence of norms implies there exist
    $0 < c_1 \leq c_2$
    such that
    $c_1 \norm{\cdot}_{*} \leq \norm{\cdot} \leq c_2 \norm{\cdot}_{*}$.
    Now set $\epsilon' = \min\{ \frac{c_1}{c_0}\epsilon_1, \frac{c_1}{c_2}\epsilon_5 \} < \epsilon_5$, where $\epsilon_1$ and $\epsilon_5$ are defined in Lemma~\ref{lemma:nlp_sensitivity} and~\ref{lemma:subproblem_sensitivity}, respectively,
    and $c_0 \coloneqq 4(\rho + \theta + \tau_x + \tau_z + 1)(\frac{1}{\rho} + 1)(\frac{1}{\sqrt{\tau_z}} + 1)(T + 1)(\norm{D} + 1)(c_2 + 1)$ is sufficiently large.
    Then, it can be verified that for all $k \geq 0$, we have:
    $
    (Dx^k, z^k, \lambda^k) \in B_{\epsilon_5} (Dx^*, z^*, \lambda^*)
    $,
    which verifies the conditions of Lemma \ref{lemma:subproblem_sensitivity} and hence of Lemma \ref{lemma:lower_bound},
    and
    $
    (p^k, d^k) \in B_{\epsilon_1}(0)
    $,
    which verifies the condition of Lemma \ref{lemma:upper_bound}.
    Now, since the conditions of Lemma \ref{lemma:iterates_and_nlp_perturbed} are also satisfied, and since $\{p^k\}$ and $\{d^k\}$ converge to $0$ (from Lemma~\ref{lemma:convergence_of_residuals}), this proves that
    $\{(x^k, z^k, \lambda^k)\} = \{(\hat{x}(p^k, d^k), \hat{z}(p^k, d^k), \hat{\lambda}(p^k, d^k))\}$ converge to $(\hat{x}(0, 0), \hat{z}(0, 0), \hat{\lambda}(0, 0)) = (x^*, z^*, \lambda^*)$ from Lemma~\ref{lemma:nlp_sensitivity}.
    Finally, the convergence rate follows from \eqref{eq:local:contraction} as follows:
    \begin{equation*}
        \frac{\norm{(Dx^k, z^k, \lambda^k, \Delta z^k) - (Dx^*, z^*, \lambda^*, 0)}^2_{\dagger}}{\norm{(Dx^{k-1}, z^{k-1}, \lambda^{k-1}, \Delta z^{k-1}) - (Dx^*, z^*, \lambda^*, 0)}^2_{\dagger}} \leq 1,
    \end{equation*}
    where we have defined the norm:
    \[
    \norm{(Dx, z, \lambda, \Delta z)}_{\dagger} \coloneqq \sqrt{\norm{(Dx, z, \lambda)}_{*}^2  + \tau_z \norm{\Delta z}^2}.
    \]
    The above is equivalent to the definition of Q-sublinear convergence \cite[Proposition~1.1d]{bertsekas2014constrained} which proves the claim.
\end{proof}

\ifCLASSOPTIONcaptionsoff
  \newpage
\fi

\bibliographystyle{IEEEtran}
\bibliography{./bibtex/bib/IEEEabrv,./bibtex/bib/bibliography}

\begin{IEEEbiographynophoto}{Anirudh Subramanyam}
    is a postdoctoral researcher in the Mathematics and Computer Science Division at Argonne National Laboratory.
    He obtained his bachelor’s degree from the Indian Institute of Technology, Bombay and his Ph.D. from Carnegie Mellon University, both in chemical engineering. His research interests are in computational methods for nonlinear and discrete optimization under uncertainty with applications in energy, transportation and process systems.
\end{IEEEbiographynophoto}
\begin{IEEEbiographynophoto}{Youngdae Kim}
    is a postdoctoral researcher in the Mathematics and Computer Science Division at Argonne National Laboratory.
    He received the B.Sc. and M.Sc. degrees in computer science and engineering from Pohang University of Science and Technology, Pohang, South Korea, and the Ph.D. degree in computer sciences from the University of Wisconsin-Madison.
    His research interests include distributed optimization methods using hardware accelerators with applications in power systems.
\end{IEEEbiographynophoto}
\begin{IEEEbiographynophoto}{Michel Schanen}
    is an assistant computer scientist in the Mathematics and Computer Science Division at Argonne National Laboratory. He received his M.Sc. in 2008 from Rheinisch-Westfälische Technische Hochschule (RWTH) Aachen, Germany and in 2014 his Ph.D. in computer science from RWTH Aachen, Germany. His interests lie in automatic differentiation and applications in large-scale computing.
\end{IEEEbiographynophoto}
\begin{IEEEbiographynophoto}{Fran\c{c}ois Pacaud}
    is a postdoctoral researcher in the Mathematics and Computer Science Division at Argonne National Laboratory.
    He obtained his M.Sc. in 2015 from Mines ParisTech, Paris, France, and in 2018 his Ph.D. in applied mathematics from the \'{E}cole des Ponts ParisTech, Paris, France. His research 
    interests encompasses stochastic and nonlinear optimization, with
    application in energy systems.
\end{IEEEbiographynophoto}
\begin{IEEEbiographynophoto}{Mihai Anitescu}
    is a senior computational mathematician in the Mathematics
    and Computer Science Division at Argonne National Laboratory and a professor in the Department of Statistics at the University of Chicago. He obtained his engineer diploma (electrical engineering) from the Polytechnic University of Bucharest in 1992 and his Ph.D. in applied mathematical and computational sciences from the University of Iowa in 1997. He specializes in the areas of numerical optimization, computational science, numerical analysis and uncertainty quantification in which he has published more than 100 papers in scholarly journals and book chapters.
    He has been recognized for his work in applied mathematics by his selection as a SIAM Fellow in 2019.
\end{IEEEbiographynophoto}
\vfill

\noindent\fbox{\parbox{\columnwidth}{\footnotesize
        The submitted manuscript has been created by UChicago Argonne, LLC, Operator of Argonne National Laboratory (``Argonne''). Argonne, a U.S. Department of Energy Office of Science laboratory, is operated under Contract No. DE-AC02-06CH11357. The U.S. Government retains for itself, and others acting on its behalf, a paid-up nonexclusive, irrevocable worldwide license in said article to reproduce, prepare derivative works, distribute copies to the public, and perform publicly and display publicly, by or on behalf of the Government. The Department of Energy will provide public access to these results of federally sponsored research in accordance with the DOE Public Access Plan (http://energy.gov/downloads/doe-public-access-plan).}
}

\end{document}